\documentclass[a4paper]{amsart} 
\usepackage{graphicx}
\usepackage{color}
\usepackage{hyperref}
\usepackage{tikz,amssymb,amsmath} 
\usepackage[arrow,curve,matrix,tips]{xy}
\usepackage{enumerate}
\usepackage{tkz-euclide}
\usepackage{esint}
\usepackage{setspace}
\usepackage{verbatim}
\makeatletter
\tikzset{
  treenode/.style = {align=center, inner sep=0pt, text centered,
    font=\sffamily},
  arn_n/.style = {treenode, circle, white, font=\sffamily\bfseries, draw=black,
    fill=black, text width=1.5em},
  arn_r/.style = {treenode, circle, red, draw=red, 
    text width=1.5em, very thick},
  arn_x/.style = {treenode, rectangle, draw=black,
    minimum width=0.5em, minimum height=0.5em}
}

\newtheorem{theorem}{Theorem}[section]
\newtheorem{lemma}[theorem]{Lemma}

\newtheorem{proposition}[theorem]{Proposition}
\newtheorem{corollary}[theorem]{Corollary}

\theoremstyle{definition} 
\newtheorem{definition}[theorem]{Definition}
\newtheorem{example}[theorem]{Example}

\newtheorem{notation}[theorem]{Notation}

\newtheorem{summary}[theorem]{Summary}
\newtheorem{remark}[theorem]{Remark}

\numberwithin{equation}{section} 
\makeatletter
\def\moverlay{\mathpalette\mov@rlay}
\def\mov@rlay#1#2{\leavevmode\vtop{%
   \baselineskip\z@skip \lineskiplimit-\maxdimen
   \ialign{\hfil$\m@th#1##$\hfil\cr#2\crcr}}}
\newcommand{\charfusion}[3][\mathord]{
    #1{\ifx#1\mathop\vphantom{#2}\fi
        \mathpalette\mov@rlay{#2\cr#3}
      }
    \ifx#1\mathop\expandafter\displaylimits\fi}
\makeatother

\begin{document} 

\title{On the Tree Structure of Orderings and Valuations on Rings}

\author{Simon M\"uller}

\begin{abstract} Let $R$ be a not necessarily commutative ring with $1.$ In the present paper we first introduce a notion of quasi-orderings, which axiomatically subsumes all the orderings and valuations on $R$. We proceed by uniformly defining a coarsening relation $\leq$ on the set $\mathcal{Q}(R)$ of all quasi-orderings on $R.$ One of our main results states that $(\mathcal{Q}(R),\leq')$ is a rooted tree for some slight modification $\leq'$ of $\leq,$ i.e. a partially ordered set admitting a maximum such that for any element there is a unique chain to that maximum. As an application of this theorem we obtain that $(\mathcal{Q}(R),\leq')$ is a spectral set, i.e. order-isomorphic to the spectrum of some commutative ring with $1.$ We conclude this paper by studying $\mathcal{Q}(R)$ as a topological space.
\end{abstract}

\begin{center} {\small{\today}} \end{center}

\maketitle

\section{Introduction} \label{sectionintro}

\noindent
Let $K$ be a field. In \cite{Efrat}, Efrat defines localities on $K$ as the union of all orderings and valuations on $K.$ Any locality $\lambda$ induces a subgroup $G_{\lambda}$ of $K^{\times},$ namely
\[
G_{\lambda} = \left\{ \begin{array}{cc} P\backslash\{0\}, & \textrm{ if } \lambda = P \textrm{ is an ordering} \\ 1+I_v, & \textrm{if } \lambda = v \textrm{ is a valuation} \end{array} \right.
\]

\noindent
where $I_v := \{x \in K\colon v(x) > 0\}.$ This observation yields a coarsening relation $\leq$ on the set $\Lambda$ of all localities on $K,$ given by $\lambda_1 \leq \lambda_2 :\Leftrightarrow G_{\lambda_2} \leq G_{\lambda_1}.$ Efrat proves that $(\Lambda,\leq)$ forms a rooted tree (cf. \cite[Corollary 7.3.6]{Efrat}), i.e. a partially ordered set admitting a minimum such that for any locality $\lambda \in \Lambda$ there is a unique chain to that minimum. \vspace{1mm}

\noindent
The main goal of the present paper is to improve this result on two levels. First of all, we generalise it from fields to not necessarily commutative rings (in this paper ring always means ring with unity). Secondly, we provide a uniform definition of $\leq$ and a uniform proof of the said tree structure. For this purpose we work with quasi-orderings. Essentially, quasi-orderings and localities are the same objects, except that the former class is defined axiomatically. Quasi-orderings were introduced in \cite{Fakh} for fields and in \cite{Muller} for commutative rings. \vspace{1mm}

\noindent
This paper is organised as follows: in section 2 we introduce quasi-orderings on rings and prove the dichotomy that they are either orderings or valuations. Section 3 deals with specific classes of quasi-orderings, namely special and Manis quasi-orderings. They provide a generalisation of the respective notions of ring valuations. In the fourth section, we uniformly define a coarsening relation $\leq$ on the set of all quasi-orderings on a ring $R,$ which generalises Efrat's one. The main goal of section 5 is to prove that, given a two-sided completely prime ideal $\mathfrak{q}$ of $R,$ the set of all quasi-orderings on $R$ with support $\mathfrak{q}$ forms a rooted tree. This investigation is continued in section 6, where we study the behaviour of special and Manis quasi-orderings with respect to $\leq.$ In the seventh section, we give applications of our tree structure theorem. We prove that the set $\mathcal{Q}(R)$ of all orderings and valuations on a ring $R$ can be realised by the spectrum of a commutative ring. Finally, in section 8, we establish that the coarsening relation $\leq$ yields a topology on $\mathcal{Q}(R).$ We conclude the paper by generalising the Harrison topology from the real spectrum to the quasi-real spectrum $\mathcal{Q}(R).$

\section{The Notion of Quasi-Ordered Rings} \label{sectionqorings}

\noindent
In his note \cite{Fakh}, Fakhruddin axiomatically introduced quasi-ordered fields $(K,\preceq).$ He established the dichotomy that any such field is either an ordered field or a valued field $(K,v)$ such that $x \preceq y \Leftrightarrow v(y) \leq v(x)$ for all $x,y \in K$ (\cite[Theorem 2.1]{Fakh}). This theorem was generalised to commutative rings in \cite{Muller} by first showing that any quasi-ordering on a commutative ring extends to its quotient field $K$ (modulo some canonical prime ideal), and then applying Fakhruddin's dichotomy to $K.$ The aim of the present section is to provide a proof of the said result that also applies to possibly non-commutative rings (Theorem \ref{dicho}). \vspace{1mm}

\begin{definition} \label{prelimi} \label{qo} A \textit{quasi-ordering} on a set $S$ is a binary, reflexive, transitive and total relation $\preceq$ on $S.$ Any quasi-ordering induces an equivalence relation $\sim_{\preceq} \, := \: \sim,$ given by
\[
\forall x,y \in S\colon x \sim y :\Leftrightarrow x \preceq y \textrm{ and } y \preceq x.
\]
We write $x \prec y$ if $x \preceq y$ but $y \not\preceq x,$ and we write $x \nsim y$ if $x \prec y$ or $y \prec x.$
\end{definition}

\begin{definition} \label{qoring2} Let $R$ be a ring, and $\preceq$ and $\sim$ as in Definition \ref{qo}.  \\
The tuple $(R,\preceq)$ is called a \textit{quasi-ordered ring} if for all $a,b,x,y,z \in R\colon$
\begin{itemize}
\item[(QR1)] $0 \prec 1,$ \vspace{1mm}

\item[(QR2)] If $0 \preceq a,b$ and  $x \preceq y,$ then $axb \preceq ayb,$ \vspace{1mm}

\item[(QR3)] If $0 \prec a,b$ and $axb \preceq ayb,$ then $ x \preceq y,$ \vspace{1mm}

\item[(QR4)] If $z \nsim y$ and $x \preceq y,$ then $x+z \preceq y+z.$ \vspace{1mm}
\end{itemize}

\noindent
The equivalence class of $0$ (w.r.t.$\: \sim$), denoted by $\mathfrak{q}_{\preceq},$ is called the \textit{support} of $\preceq.$
\end{definition}

\noindent
For the rest of this section, if not mentioned otherwise, $(R,\preceq)$ always denotes a quasi-ordered ring.

\begin{lemma} \label{lem1} Let $x,y \in R$ such that $x \nsim 0$ and $y \sim 0.$ Then $x+y \sim x.$ In particular, $z \sim 0$ if and only if $-z \sim 0$ for any $z \in R.$ 
\end{lemma}
\begin{proof}
Since $x \nsim 0$ and $y \sim 0,$ we obtain $x \nsim y.$ So from $y \preceq 0$ we deduce that $x+y \preceq x,$ and from $0 \preceq y$ we deduce that $x \preceq x+y.$ Thus, $x+y \sim x.$ \vspace{1mm}

\noindent 
If $z \sim 0,$ then $-z \nsim 0$ implies $0 = (-z)+z \sim -z \nsim 0,$ a contradiction. Therefore, $-z \sim 0.$ The converse is obviously also true due to symmetry reasons.
\end{proof}

\begin{lemma} \label{-+} If $x \preceq 0,$ then $0 \preceq -x.$
\begin{proof} Let $x \preceq 0.$ If $-x \prec 0,$ then $x \preceq 0$ implies via (QR4) that $0 \preceq -x,$ a contradiction. Hence, $0 \preceq -x.$
\end{proof}
\end{lemma}

\begin{proposition} \label{prime} The support $\mathfrak{q}_{\preceq}$ is a two-sided completely prime ideal of $R.$
\begin{proof} Clearly, $0 \in \mathfrak{q}_{\preceq}.$ Now let $x,y \in \mathfrak{q}_{\preceq},$ and assume that $x+y \notin \mathfrak{q}_{\preceq}.$ So we have $x+y \nsim 0,$ but $-y \sim 0$ by Lemma \ref{lem1}. The same lemma yields the contradiction $$0 \sim x = (x+y)-y \sim x+y \nsim 0.$$ 
Next, suppose that $x \in \mathfrak{q}_{\preceq}$ and $r \in R.$ If $0 \preceq r,$ then evidently $xr \in \mathfrak{q}_{\preceq}$ by (QR2). If $r \preceq 0,$ then Lemma \ref{lem1} and Lemma \ref{-+} imply $0 \preceq -r$ and $-x \sim 0.$ Therefore, $xr=(-x)(-r) \sim 0,$ so again $xr \in \mathfrak{q}_{\preceq}.$ Analogously, it follows that $rx \in \mathfrak{q}_{\preceq}.$ Thus, $\mathfrak{q}_{\preceq}$ is a two-sided ideal of $R.$ \vspace{1mm}

\noindent 
It remains to show that $\mathfrak{q}_{\preceq}$ is completely prime. Axiom (QR1) states that $1 \notin \mathfrak{q}_{\preceq},$ i.e. $\mathfrak{q}_{\preceq} \neq R.$ Finally, assume that $xy \in \mathfrak{q}_{\preceq},$ but $x,y \notin \mathfrak{q}_{\preceq}.$ By Lemma \ref{lem1}, also $-xy \in \mathfrak{q}_{\preceq}.$ So we may w.l.o.g. assume that $0 \prec x,$ as if $x \prec 0,$ then $0 \prec -x$ (Lemma \ref{-+}). But if $0 \prec x,$ then (QR3) yields the contradiction $y \in \mathfrak{q}_{\preceq}.$ Hence, $\mathfrak{q}_{\preceq}$ is completely prime. 
\end{proof}
\end{proposition}

\noindent
We first show that the quasi-orderings $\preceq$ on $R$ such that $-1 \prec 0$ are precisely the orderings on $R.$

\begin{lemma} \label{sign} Let $x \in R$ and suppose that $-1 \prec 0.$ Then $0 \preceq x$ if and only if $-x \preceq 0.$ Moreover, either $x \prec -x$ or $-x \prec x$ for any $x \in R\backslash\mathfrak{q}_{\preceq}.$
\begin{proof} If $x \preceq 0,$ then $0 \preceq -x$ by Lemma \ref{-+}. On the other hand, if $0 \preceq x,$ then $-x \preceq 0$ by the assumption $-1 \prec 0$ and axiom (QR2). The second assertion follows immediately from the first one and the fact that $x,-x \nsim 0.$
\end{proof}
\end{lemma}

\begin{definition}  \label{oring} Let $\leq$ be a binary, reflexive, transitive and total relation on $R.$ Then $(R,\leq)$ is called an \textit{ordered ring} if for all $a,b,x,y,z \in R\colon$ 
\begin{itemize}
 \item[(O1)] $0 < 1,$ \vspace{1mm}

 \item[(O2)] If $0 \leq a,b$ and $x \leq y,$ then $axb \leq ayb,$ \vspace{1mm}

 \item[(O3)] If $xy \leq 0,$ then $x \leq 0 \textrm{ or } y \leq 0,$ \vspace{1mm}

 \item[(O4)] If $x \leq y,$ then $x+z \leq y+z.$
 \end{itemize}
\end{definition}

\begin{remark} The corresponding positive cone $P:= \{x \in R\colon 0 \leq x\}$ fulfils the axioms $P+P \subseteq P, PP \subseteq P,$ $P \cup -P = R$ and that $P \cap -P =: \mathfrak{q}_P$ is a completely prime ideal of $R.$
\end{remark}

\begin{proposition} \label{dichoord1} Any ordered ring is a quasi-ordered ring.
\begin{proof} We only have to verify (QR3). So suppose that $0 < a,b$ and $axb \leq ayb.$ The latter implies $a(x-y)b \leq 0.$ Hence, from (O3) and the assumption on $a$ and $b,$ it follows that $x-y \leq 0,$ i.e. $x \leq y.$
\end{proof}
\end{proposition}

\begin{proposition} \label{dichoord2} Any quasi-ordered ring $(R,\preceq)$ with $-1 \prec 0$ is an ordered ring.
\begin{proof} It suffices to verify axiom (O4). So suppose that $x \preceq y.$ This implies that $x-y \preceq 0.$ Indeed, either $-y \sim 0$ and then $x-y \sim x \preceq y \sim 0$ by Lemma \ref{lem1}; or $-y \not\sim 0$ and then $y \not\sim -y$ by Lemma \ref{sign}, so $x-y \preceq 0$ follows from $x \preceq y$ via axiom (QR4). The same reasoning shows that $x-y \preceq 0$ implies $x+z \preceq y+z.$ 
\end{proof}
\end{proposition}

\noindent
It remains to show that the quasi-orderings with $0 \prec -1$ correspond to valuations. 

\begin{definition} \label{defvalring2} Let $(\Gamma,+,\leq)$ be a totally ordered, cancellative monoid, and $\infty$ a symbol such that $\infty > \Gamma$ and $\infty = \infty + \infty = \gamma + \infty = \infty + \gamma$ for all $\gamma \in \Gamma.$ \\
A surjective map $v\colon R \to \Gamma \cup \{\infty\}$ is called a \textit{valuation} on $R$ if for all $x,y \in R\colon$
\begin{itemize}
 \item[(V1)] $v(0) = \infty,$ \vspace{1mm}

 \item[(V2)] $v(1) = 0,$ \vspace{1mm}

 \item[(V3)] $v(xy) = v(x)+v(y),$ \vspace{1mm}

 \item[(V4)] $v(x+y) \geq \min\{v(x),v(y)\}.$
\end{itemize}
We call $\Gamma_v := \Gamma$ the \textit{value group} of $v,$ and $\mathfrak{q}_v := v^{-1}(\infty)$ the \textit{support} of $v.$
\end{definition}

\begin{remark} \label{remonval} Valuations on commutative rings are usually defined by demanding that $v$ maps into the group generated by $v(R\backslash\mathfrak{q}_v).$ In fact, these two definitions of a valuation are equivalent. If $R$ is commutative and if $v$ is a valuation on $R$ in the sense of Definition \ref{defvalring2}, then $\Gamma_v$ embeds in an ordered group (\cite[Corollary 1]{Keha}). Conversely, $v(R\backslash\mathfrak{q}_v)$ is a monoid for any map $v$ on $R$ that satisfies the axioms (VR1)\,-\,(VR4). \vspace{1mm}

\noindent
However, as was proven by Mal'cev, Chihata and Vinogradov, there are ordered cancellative monoids which cannot be embedded in any ordered group (cf. \cite[Theorem 9.8]{Lam2}). Consequently, whenever we consider possibly non-commutative valued rings, we stick to Definition \ref{defvalring2}. For valuations on commutative rings, we occasionally suppose that $v$ maps into the group $\Gamma_v$ for the sake of convenience.
\end{remark}

\begin{proposition} \label{dichoval1}
Let $(R,v)$ be a valued ring. Then $(R,\preceq)$ is a quasi-ordered ring with support $\mathfrak{q}_v,$ where 
$x \preceq y :\Leftrightarrow v(y) \leq v(x)$ for all $x,y \in R.$
\begin{proof}
Clearly, $\preceq$ is reflexive, transitive and total, since the ordering $\leq$ on $\Gamma$ has these properties. Further note that $v(1) = 0 < \infty = v(0),$ thus $0 \prec 1.$ This shows that (QR1) is fulfilled. Next, we establish (QR2). From $x \preceq y$ it follows that $v(y) \leq v(x).$ Hence, $$v(ayb) = v(a)+v(y)+v(b) \leq v(a)+v(x)+v(b) = v(axb),$$ and therefore $axb \preceq ayb.$ For the verification of (QR3), let $0 \prec a,b$ and suppose that $y \prec x,$ i.e. $v(a),v(b) < \infty$ and $v(x) < v(y).$ Then $v(axb) \leq v(ayb).$ Note that equality cannot hold, because then it follows via (V3) and cancellation in $\Gamma$ that $v(x) = v(y),$ a contradiction. Hence, 
$v(axb) < v(ayb),$ and therefore $ayb \prec axb.$ For the proof of (QR4) we refer to the commutative case, cf. \cite[Lemma 3.7]{Muller}. \vspace{1mm}

\noindent
Finally, $x \in \mathfrak{q}_{\preceq}$ if and only if $v(x) = v(0) = \infty,$ i.e. if and only if $x \in \mathfrak{q}_v.$
\end{proof}
\end{proposition}

\noindent
It remains to prove that any quasi-ordering $\preceq$ on $R$ with $0 \prec -1$ is induced by a valuation $v$ on $R.$ So suppose for the rest of this section that $0 \prec -1.$

\begin{lemma} \label{positive} We have $0 \preceq x$ for all $x \in R.$ 
\begin{proof} Suppose that $x \prec 0$ for some $x.$ Then $0 \preceq -x$ by Lemma \ref{-+}, so from the assumption $0 \prec -1$ and axiom (QR2) it follows $0 \preceq x,$ a contradiction.
\end{proof}
\end{lemma}

\begin{corollary} \label{cor10} \hspace{7cm}
\begin{enumerate}
\item If $x \preceq y,$ then $ax \preceq ay$ for all $a,b,x,y \in R.$ \vspace{1mm}

\item $x \sim -x$ for all $x \in R.$
\end{enumerate}
\end{corollary}
\begin{proof} (1) is trivial by Lemma \ref{positive} and (QR2). For (2), note that $-1 \preceq 1$ or $1 \preceq -1$ by the totality of $\preceq.$ Applying (QR2) with $0 \prec -1$ yields that the converse also holds, whence $-1 \sim 1.$ Since $\sim$ is preserved under multiplication according to (1), this implies $-x \sim x$ for all $x \in R.$
\end{proof}

\noindent
Let $H:= (R\backslash\mathfrak{q}_{\preceq})/{\sim}.$ We equip $H$ with the addition $[x]+[y] := [xy]$ and the ordering $[x] \leq [y] :\Leftrightarrow y \preceq x.$

\begin{lemma} \label{cancel} $(H,+,\leq)$ is a totally ordered, cancellative monoid.
\begin{proof} We first show that $+$ is well-defined. So suppose that $[x] = [x']$ and $[y] = [y'],$ i.e. $x \sim x'$ and $y \sim y'.$ From $x \sim x'$ it follows $xy' \sim x'y',$ and from $y \sim y'$ it follows $xy \sim xy'$ by Corollary \ref{cor10}(1). Hence, $xy \sim xy' \sim x'y',$ and therefore $$[x]+[y] = [xy] = [x'y'] = [x']+[y'].$$

\noindent
By axiom (QR1), $1 \in R\backslash \mathfrak{q}_{\preceq},$ and $[1] \in H$ is easily seen to be the neutral element of $H.$ It is also clear that $+$ is associative, whence $(H,+)$ is a monoid. \vspace{1mm}

\noindent
\noindent
Next, we show that $H$ is totally ordered by $\leq.$ We first prove that $\leq$ is well-defined. So let $[x] \leq [y],$ and let $a \in [x]$ and $b \in [y],$ i.e. $a \sim x$ and $b \sim y.$ Then 
$$[x] \leq [y] \Leftrightarrow y \preceq x \Leftrightarrow b \preceq a \Leftrightarrow [a] \leq [b].$$
It is clear that $\leq$ is reflexive, transitive and total. Moreover, $\leq$ is anti-symmetric because if $[x] \leq [y]$ and $[y] \leq [x],$ then $x \sim y,$ i.e. $[x] = [y].$ Now suppose that $[x] \leq [y].$ Then $y \preceq x,$ whence $ayb \preceq axb$ by Corollary \ref{cor10}(1). Consequently, $$[a]+[x]+[b] = [axb] \leq [ayb] = [a]+[y]+[b].$$

\noindent
Finally, suppose that $[axb] \leq [ayb].$ Then $ayb \preceq axb,$ so Corollary \ref{cor10}(1) and (QR3) yield $x \sim y,$ i.e. $[x] = [y].$ Hence, $(H,\leq)$ is cancellative.
\end{proof}
\end{lemma}

\noindent
We we will verify in Proposition \ref{dichoval2} that the map 
\begin{align*} v\colon R \to H \cup \{\infty\}, \, x \mapsto \left\{ \begin{array}{cc} [x], & x \in R\backslash\mathfrak{q}_{\preceq} \\ \infty, & x \in \mathfrak{q}_{\preceq} \end{array} \right.
\end{align*}
is a valuation on $R$ that induces $\preceq.$

\begin{lemma} Let $(R,\preceq)$ be a quasi-ordered ring. Then $x+y \preceq \max\{x,y\}$ for all $x,y \in R.$ \label{lemmaineq}
\end{lemma} 
\begin{proof}
As in the field case, see \cite[Lemma 4.1]{Fakh}. 
\end{proof}

\begin{proposition} \label{dichoval2} If $(R,\preceq)$ is a quasi-ordered ring with $0 \prec -1,$ then there is a unique valuation $v$ on $R$ (up to equivalence) such that $x \preceq y \Leftrightarrow v(y) \leq v(x).$ Moreover, $\mathfrak{q}_{\preceq} = \mathfrak{q}_v.$
\begin{proof} We have already proven that $\Gamma_v = H$ is an ordered, cancellative monoid (Lemma \ref{cancel}). By construction, $v$ is obviously surjective and satisfies $v(0) = \infty$ and $v(1) = [1] = 0_H.$ Thus, (V1) and (V2) are fulfilled. For (V3) note that $xy \in \mathfrak{q}_{\preceq}$ if and only if $x \in \mathfrak{q}_{\preceq}$ or $y \in \mathfrak{q}_{\preceq}$ (Proposition \ref{prime}), so we w.l.o.g. assume that $x,y \notin \mathfrak{q}_{\preceq}.$ So $v(xy) = [xy] = [x]+[y] = v(x)+v(y),$ as desired. The axiom (V4) is immediately implied by the previous lemma. Hence, $v$ is a valuation. \vspace{2mm}

\noindent
Moreover (again w.l.o.g. $x,y \in R\backslash \mathfrak{q}_{\preceq}$), $$x \preceq y \Leftrightarrow [y] \leq [x] \Leftrightarrow v(y) \leq v(x),$$
and if $w$ is another such valuation, then $w(y) \leq w(x)$ if and only if $v(y) \leq v(x),$ i.e. $v$ and $w$ are equivalent. Finally, $$x \in \mathfrak{q}_v \Leftrightarrow v(x) = v(0) \Leftrightarrow x \sim 0 \Leftrightarrow x \in \mathfrak{q}_{\preceq},$$
i.e. $\mathfrak{q}_v = \mathfrak{q}_{\preceq}.$ This finishes the proof.
\end{proof}
\end{proposition}

\noindent
The Propositions \ref{dichoord1}, \ref{dichoord2}, \ref{dichoval1} and \ref{dichoval2} imply:

\begin{theorem} \label{dicho} Let $R$ be a ring and let $\preceq$ be a binary relation on $R.$ Then $(R, \preceq)$ is a quasi-ordered ring if and only if it is either an ordered ring or there is a unique valuation $v$ on $R$ (up to equivalence) such that $x \preceq y \Leftrightarrow v(y) \leq v(x).$ \\ Moreover, the support of the quasi-ordering coincides with the support of the ordering, respectively with the support of the valuation.
\end{theorem}

\section{Special and Manis Quasi-Orderings} \label{sectionspecialmanis}

\noindent
We continue our investigation by introducing special and Manis quasi-orderings (Definition \ref{specialqo}, respectively Definition \ref{manisqo}). These quasi-orderings subsume certain kinds of orderings and valuations. Special quasi-orderings yield a uniform approach to special valuations on the one hand, and orderings that are maximal elements in the real spectrum of the given ring on the other hand. Manis quasi-orderings provide a unification of Manis valuations and orderings with maximal support, if the ring is commutative. In general, they subsume Manis valuations with support $\{0\}$ and skew-field orderings (Theorem \ref{dichomanis1}). \vspace{1mm}

\noindent
Manis valuations were introduced in \cite{Manis} by Manis himself on commutative rings. These are valuations such that $v(R\backslash \mathfrak{q}_{\preceq})$ is a group, like for example any valuation on a field. Amongst other things, Manis used them to establish an approximation theorem for valuations on rings. Manis valuations play a prominent role in the theory of valued rings (cf. e.g. \cite{Kneb} and \cite{Kneb2}).

\begin{definition} (cf. [\cite{Manis}]) \label{Manis} Let $R$ be a ring. A valuation $v$ on $R$ is called a \textit{Manis valuation}, if its value monoid $\Gamma_v$ is a group.
\end{definition}

\noindent
Special valuations implicitly first appeared in the work of Huber and Knebusch on valuation spectra. These are precisely the valuations that admit no proper primary specialisation in the valuation spectrum of a commutative ring (cf. \cite{Huber}). They were formally introduced later on by Knebusch and Zhang in \cite{Kneb}.

\begin{definition} \label{defspecialval} (\cite[Chapter I]{Kneb}) Let $R$ be a commutative ring and $v$ a valuation on $R$ with value group $\Gamma$ (see Remark \ref{remonval}). 
\begin{enumerate}
\item The \textit{characteristic subgroup} $c_v(\Gamma)$ of $\Gamma$ is the smallest convex subgroup of $\Gamma$ containing all elements $v(x)$ such that $x \in R$ and $v(x) \leq 0,$ i.e.
\[
\quad \quad \quad c_v(\Gamma) = \{\gamma \in \Gamma\colon v(x) \leq \gamma \leq -v(x) \textrm{ for some } x \in R \textrm{ with } v(x) \leq 0\}.
\]
\item $v$ is called \textit{special}, if $c_v(\Gamma) = \Gamma.$
\end{enumerate}
\end{definition}

\noindent
In order to deduce a notion of special valuations on possibly non-commutative rings we make use of a different characterisation.

\begin{lemma} \label{special} $\mathrm{(cf.} \, $\cite[Proposition I.2.2]{Kneb}$\mathrm{)}$ Let $R$ be a commutative ring and let $v$ be a valuation on $R.$ Then 
\begin{enumerate}
\item $\mathfrak{q}_v \subseteq \{r \in R\colon v(xr) > 0 \textrm{ for all } x \in R\}$ \vspace{1mm}

\item The following are equivalent: \vspace{1mm}

\begin{itemize}
\item[(i)] $v$ is special, \vspace{1mm}

\item[(ii)] $\mathfrak{q}_v = \{r \in R\colon v(xr) > 0 \textrm{ for all } x \in R\}.$
\end{itemize}
\end{enumerate}
\end{lemma}
\begin{proof} The proof of (1) is trivial. For (2) first suppose that $v$ is special, and let $r \in R$ such that $v(xr) = v(x) + v(r) > 0$ for all $x \in R.$ This implies $v(r) > -v(x)$ for all $x \in R.$ Thus, $v(r) \notin c_v(\Gamma).$ Since $v$ is special, it follows $r \in \mathfrak{q}_v.$ \vspace{1mm}

\noindent
For the converse, suppose that (ii) holds and assume that $v$ is not special. Then we find some $\gamma \in \Gamma \backslash c_v(\Gamma),$ w.l.o.g. $\gamma \in v^{-1}(R\backslash q_v).$ Let $r \in R$ such that $v(r) = \gamma.$ Since $\gamma \notin c_v(\Gamma),$ we know that $v(r) > 0.$ We claim that $v(xr) > 0$ for all $x \in R.$ By (ii), this yields $r \in \mathfrak{q}_v,$ the desired contradiction. So let $x \in R,$ w.l.o.g. $v(x) < 0 < v(r).$ From $v(r) \notin c_v(\Gamma),$ it follows $-v(x) < v(r).$ Hence, $0 < v(xr).$
\end{proof}

\noindent
We further show that we do not have to distinguish between left and right special valuations.
 
\begin{lemma} \label{lemmaspecial}
Let $R$ be a ring, $v$ a valuation on $R$ and $x,y \in R.$ Then $v(xy) > 0$ if and only if $v(yx) > 0.$
\end{lemma}
\begin{proof} We prove the lemma by contraposition. So let $x,y \in R,$ w.l.o.g. $x,y \notin \mathfrak{q}_v.$ By cancellation in $\Gamma_v,$ we obtain
\[
v(x)+v(y) \leq 0 \Leftrightarrow v(y)+v(x)+v(y) \leq v(y) \Leftrightarrow v(y)+v(x) \leq 0. 
\]
\end{proof}

\begin{definition} \label{specialnc} Let $R$ be a ring and let $v$ be a valuation on $R.$ We call $v$ \textit{special}, if $\mathfrak{q}_v = \{r \in R\colon v(xr) > 0 \textrm{ for all } x \in R\}.$
\end{definition}

\begin{remark} \label{remonval2} Let $R$ be a ring.
\begin{enumerate}
\item By definition and Lemma \ref{lemmaspecial}, a valuation $v$ on $R$ is \vspace{1mm}

\begin{enumerate}
\item special, if and only if for any $x \in R\backslash\mathfrak{q}_v$ we find some $y \in R$ such that $v(xy) \leq 0,$ or equivalently, $v(yx) \leq 0.$ \vspace{1mm}

\item Manis, if and only if for any $x \in R\backslash\mathfrak{q}_v$ we find some $y \in R$ such that $v(xy) = 0,$ or equivalently, $v(yx) = 0.$
\end{enumerate} \vspace{1mm}

\item Evidently, any Manis valuation is a special valuation.
\end{enumerate}
\end{remark}

\noindent
We first deduce a notion of special quasi-orderings. By Remark \ref{remonval2}(1), a quasi-ordering shall be called special if and only if for any $x \in R\backslash \mathfrak{q}_v$ there is some $y \in R$ such that $1 \preceq xy,$ respectively $1 \preceq yx.$ As in the case of special valuations, left and right special quasi-orderings coincide (see Proposition \ref{leftrightspecial}).

\begin{lemma} \label{QR2-} Let $(R,\preceq)$ be a quasi-ordered ring and let $x,y,z \in R.$ If $x \preceq y$ and $z \preceq 0,$ then $yz \preceq xz$ and $zy \preceq zx.$
\begin{proof} For symmetry reasons, it suffices to show that $yz \preceq xz.$ We may assume that $z \notin \mathfrak{q}_{\preceq}.$ Moreover, if $y \in \mathfrak{q}_{\preceq},$ then $x,z \preceq 0,$ whence $0 \preceq -x,-z$ by Lemma \ref{-+}. Thus, by (QR2), $yz \sim 0 \preceq xz.$ So we may further assume that $y \notin \mathfrak{q}_{\preceq}.$ Now $x \preceq y$ and $0 \preceq -z$ implies $-xz \preceq -yz.$ We claim that $yz \nsim -yz.$ Once this is shown, it follows from $-xz \preceq -yz$ that $yz-xz \preceq 0.$ The latter implies $yz \preceq xz.$ Indeed, either $x \nsim 0$ and therefore $xz \nsim 0$ (Proposition \ref{prime}), so that we can apply (QR4); or $x \sim 0,$ and therefore $yz - xz \sim yz \preceq 0 \sim xz$ (Lemma \ref{lem1}). \vspace{1mm}

\noindent 
Assume for the sake of a contradiction that $yz \sim -yz.$ This yields $0 \preceq yz, -yz$ (Lemma \ref{-+}). As $y \notin \mathfrak{q}_{\preceq},$ either $0 \prec y$ or $0 \prec -y.$ So via (QR3) we deduce from either $0 \preceq yz$ (if $0 \prec y$) or $0 \preceq -yz$ (if $0 \prec -y$) that $0 \preceq z.$ Hence $z \sim 0,$ the desired contradiction. 
\end{proof}
\end{lemma}

\begin{lemma} \label{QR3-} Let $(R,\preceq)$ be a quasi-ordered ring and let $x,y,z \in R.$ If $xz \preceq yz$ (respectively $zx \preceq zy$) and $z \prec 0,$ then $y \preceq x.$
\begin{proof} Suppose that $xz \preceq yz$ and $z \prec 0.$ Assume that $x \prec y.$ Then Lemma \ref{QR2-} yields $yz \preceq xz,$ whence $xz \sim yz.$ Since $0 \prec -z$ (Lemma \ref{-+}), we obtain via (QR3) that $-x \sim -y.$ By (QR2) (if $0 \prec -1$), respectively Lemma \ref{QR2-} (if $-1 \prec 0$), it follows that $x \sim y,$ a contradiction. Thus, $y \preceq x.$
\end{proof}
\end{lemma}

\begin{proposition} $\mathrm{(see \; Lemma \; \ref{lemmaspecial})}$ \label{leftrightspecial} Let $(R,\preceq)$ be a quasi-ordered ring and let $x,y \in R.$ Then $1 \preceq xy$ if and only if $1 \preceq yx.$
\begin{proof} Let $1 \preceq xy.$ If $x \prec 0,$ then Lemma \ref{QR2-} implies $xyx \preceq x.$ Consequently, Lemma \ref{QR3-} yields $1 \preceq yx.$ If $0 \prec x,$ then we just apply axiom (QR2) instead of Lemma \ref{QR2-}, and axiom (QR3) instead of Lemma \ref{QR3-}.
\end{proof}
\end{proposition}

\begin{definition} \label{specialqo} Let $R$ be a ring. We call a quasi-ordering $\preceq$ on $R$ \textit{special}, if for any $x \in R\backslash\mathfrak{q}_{\preceq}$ there is some $y \in R$ such that $1 \preceq xy$ (or equivalently, $1 \preceq yx$).
\end{definition}

\noindent
Definition \ref{specialqo} automatically yields a notion of special orderings. Remarkably, it coincides with Krivine's notion of \textit{pseudo-fields} (\cite{Krivine}, see also \cite[Definition 13.2.6]{Dick}). Krivine showed that a preordering $T$ on a real ring $R$ is a maximal preordering if and only if $R/T$ is a pseudo-field (\cite[Theorem 6]{Krivine}). Special orderings are precisely the maximal elements in the real spectrum of a given commutative ring (cf. \cite[Theorem 13.2.8]{Dick}).

\begin{example} \label{expspecial} \hspace{7cm}

\begin{enumerate}
\item (cf. \cite[Ex. 17.3]{Lam4}) Consider the Weyl algebra $$R := \mathbb{R}\langle X,Y\rangle/(XY-YX-1).$$ We write $x$ and $y$ for the residue classes of $X$ and $Y,$ respectively. Since $yx = xy-1,$ any $r \in R$ can be written in the canonical form
\[
r = r_0(x)+{r_1(x)y} + \ldots + r_n(x)y^n
\]
for some $n \in \mathbb{N}_0,$ where $r_i \in \mathbb{R}[X]$ for all $0 \leq i \leq n.$ The chain
\[
\quad \quad \quad \mathbb{R} < x < x^2 < \ldots < y < xy < x^2y < \ldots < y^2 < xy^2 < x^2y^2 < \ldots
\]
defines a special ordering on $R.$ \vspace{1mm}

\item The unique ordering on $R= \mathbb{R}[X]$ such that $0 < X < \mathbb{R}^{>0}$ is not special, since $|Xr| < 1$ for any $r \in R.$ 
\end{enumerate}
\end{example}

\noindent
Likewise, we may introduce Manis quasi-orderings. By Remark \ref{remonval2}(1), a quasi-ordering shall be called Manis if and only if for any $x \in R\backslash \mathfrak{q}_v$ there is some $y \in R$ such that $1 \sim xy,$ respectively $1 \sim yx.$ As for special quasi-orderings, the notions of left and right Manis quasi-orderings coincide (see Corollary \ref{Manisdef}).

\begin{lemma} \label{lemmacorpres} Let $(R,\preceq)$ be a quasi-ordered ring and let $x,y \in R.$ If $x \sim y,$ then $x \sim -y$ or $0 \sim x-y.$
\begin{proof} If $x,y \sim 0,$ then $x \sim -y,$ as $\mathfrak{q}_{\preceq}$ is an ideal. So suppose that $x,y \nsim 0,$ and assume that $x \nsim -y.$ We show that $0 \sim x-y.$ Note that $y \preceq x \nsim -y.$ Thus, $0 \preceq x-y$ by (QR4). Likewise, $x \preceq y \sim x \nsim -y,$ so $y \nsim -y,$ and therewith $x-y \preceq 0.$ Hence, $0 \sim x-y.$ 
\end{proof}
\end{lemma}

\begin{corollary} \label{corpres} Let $(R,\preceq)$ be a quasi-ordered ring. Then $\sim$ is preserved under multiplication, i.e. if $x \sim y,$ then also $axb \sim ayb$ for all $x,y,a,b \in R.$ 
\begin{proof} We show that $x \sim y$ implies $ax \sim ay,$ the proof of right multiplication being analogue. The cases $0 \preceq a$ (axiom (QR2)) and $x,y$ in $\mathfrak{q}_{\preceq}$ ($\mathfrak{q}_{\preceq}$ is an ideal) are both trivial. So suppose that $0 \npreceq a$ and $x,y \nsim 0.$ Then $0 \preceq -a.$ Lemma \ref{lemmacorpres} gives rise to a proof by cases. \vspace{1mm}

\noindent
First suppose $0 \sim x-y.$ Since $-x\nsim 0,$ we have $-x \nsim x-y.$ Hence, $0 \sim x-y$ yields $-x \sim -y$ by (QR4). Because of $0 \preceq -a,$ axiom (QR2) implies $ax \sim ay.$ \vspace{1mm}

\noindent
Now suppose $0 \nsim x-y.$ Then also $0 \nsim y-x.$ Thus, Lemma \ref{lemmacorpres} implies $x \sim -y$ and $y \sim -x.$ Therefore, $-x \sim y \sim x \sim -y.$ So we have $-x \sim -y$ and $0 \preceq -a,$ whence the claim follows by applying (QR2).
\end{proof}
\end{corollary}

\begin{corollary} \label{QR3+} Let $(R,\preceq)$ be a quasi-ordered ring. Then $\sim$ is preserved under cancellation, i.e. if $axb \sim ayb$ and $0 \not\sim a,b,$ then also $x \sim y$ for all $x,y,a,b \in R.$
\begin{proof} Due to symmetry, we may w.l.o.g. assume that $b = 1.$ If $0 \prec a,$ the claim follows immediately by (QR3). If $a \prec 0,$ then $0 \prec -a$ (Lemma \ref{-+}), so $ax \sim ay$ implies $-x \sim -y.$ By Corollary \ref{corpres} we get $x \sim y.$
\end{proof}
\end{corollary}

\begin{corollary} \label{Manisdef} Let $(R,\preceq)$ be a quasi-ordered ring and let $x,y \in R.$ Then $xy \sim 1$ if and only if $yx \sim 1.$ 
\begin{proof} Due to symmetry, it suffices to show that $xy \sim 1$ implies $yx \sim 1.$ So let $xy \sim 1.$ Then Corollary \ref{corpres} yields $yxy \sim y.$ Since $xy \sim 1,$ we know that $y \nsim 0.$ Therefore, $yxy \sim y$ implies $yx \sim 1$ (Corollary \ref{QR3+}).
\end{proof}
\end{corollary}

\begin{definition} \label{manisqo} Let $R$ be a ring. A quasi-ordering $\preceq$ on $R$ is called \textit{Manis}, if for any $x \in R\backslash\mathfrak{q}_{\preceq}$ there is some $y \in R$ such that $1 \sim xy$ (or equivalently, $1 \sim yx$).
\end{definition}

\noindent
Obviously, any Manis quasi-ordering is special. Manis quasi-orderings give rise to the following dichotomy:

\begin{theorem} \label{dichomanis1} Let $(R,\preceq)$ be a quasi-ordered ring. Then $\preceq$ is Manis if and only if either $(R/\mathfrak{q}_{\preceq},\preceq)$ is an ordered division ring or there is a unique Manis valuation $v$ on $R$ (up to equivalence) such that $x \preceq y \Leftrightarrow v(y) \leq v(x)$ for all $x,y \in R.$
\begin{proof} By Theorem \ref{dicho}, $\preceq$ is either an ordering or induced by a unique valuation $v$ on $R.$ In the latter case, $v$ is obviously a Manis valuation by Remark \ref{remonval2}(1) and Definition \ref{manisqo}. So suppose from now on that $\preceq$ is an ordering on $R.$ Hence, for all $x,y \in R$ we have $x \sim y$ if and only if $x = y+c$ for some $c \in \mathfrak{q}_{\preceq}.$ \vspace{1mm}

\noindent
First suppose that $\preceq$ satisfies the Manis property, and let $\overline{0} \neq \overline{x} \in R/\mathfrak{q}_{\preceq}.$ Then we find some $y \in R$ such that $xy \sim 1,$ i.e. $xy = 1+c$ for some $c \in \mathfrak{q}_{\preceq}.$ Hence, $\overline{xy} = \overline{1}.$ Analogously, we obtain $\overline{yx} = \overline{1}.$ Therefore, $(R/\mathfrak{q}_{\preceq},\preceq)$ is an ordered division ring. \vspace{1mm}

\noindent
Conversely, let $R/\mathfrak{q}_{\preceq}$ be a division ring and let $\overline{0} \neq \overline{x} \in R/\mathfrak{q}_{\preceq}.$ Then we find some $y \in R/\mathfrak{q}_{\preceq}$ such that $\overline{xy} = \overline{1},$ whence $xy \sim 1$ by Lemma \ref{lem1}. Thus, $\preceq$ has the Manis property.
\end{proof}
\end{theorem}

\noindent
Consequently, if $R$ is a commutative ring, a Manis quasi-ordering on $R$ is either a Manis valuation or an ordering whose support is a maximal ideal of $R.$

\begin{example} \label{expmanis} \hspace{7cm}
\begin{enumerate}
\item Consider the commutative ring $R = C([0,1],\mathbb{R})$ of all continuous maps $f\colon [0,1] \to \mathbb{R}$ with pointwise addition and multiplication, and let $x \in [0,1].$ Then $f \leq g :\Leftrightarrow f(x) \leq g(x)$ defines a Manis ordering on $R.$ Indeed, if $f \in R$ such that $f(x) = r \neq 0,$ then the constant map $g(x) = 1/r$ is also in $R$ and $fg \sim 1.$ \vspace{1mm}

\item The ordering on the Weyl algebra $R$ from Example \ref{expspecial}(1) is not Manis, since for any $r \in R\backslash\{0\}$ either $ry < 0$ or $ry > 1.$
\end{enumerate}
\end{example}

\section{A Coarsening Relation on the Quasi-Orderings of a Ring} \label{sectioncoarser}

\noindent
Quasi-orderings serve the purpose of treating orderings and valuations on rings simultaneously and uniformly. This is demonstrated in the present section, where we uniformly define a coarsening relation on the set of all quasi-orderings on a ring (Definition \ref{defcoarse}), that way subsuming three different notions at once: inclusions of orderings, coarsenings of valuations, and compatibility of orderings and valuations. This coarsening relation coincides with the one that Efrat introduced in \cite[Ch. 7]{Efrat} for localities on fields and plays a key role in all of the following sections. Here, we prove that it is a partial ordering (Proposition \ref{tree1}). \vspace{1mm} 

\begin{notation} If $R$ is a ring, we denote by $\mathcal{Q}(R)$ the set of all its quasi-orderings.
\end{notation}

\noindent
The perhaps most natural approach for a coarsening relation $\leq$ on $\mathcal{Q}(R)$ would be to declare that $\preceq_2$ is coarser than $\preceq_1$ if and only if $x \preceq_1 y$ implies $x \preceq_2 y,$ i.e. if and only if $\preceq_1 \; \subseteq \; \preceq_2.$ Note, however, that in this case $\preceq_1$ is an ordering if and only if $\preceq_2$ is an ordering, since the sign of $-1$ would be the same for both quasi-orderings. A slight modification yields:

\begin{definition} \label{defcoarse} Let $R$ be a ring, and let $\preceq_1$ and $\preceq_2$ be quasi-orderings on $R.$ We say that $\preceq_2$ is \textit{coarser} than $\preceq_1$ (or $\preceq_1$ \textit{finer} than $\preceq_2$), written $\preceq_1 \leq \preceq_2,$ if 
\[
\forall x,y \in R\colon 0 \preceq_1 x \preceq_1 y \Rightarrow x \preceq_2 y.
\]
Two quasi-orderings $\preceq_1, \preceq_2$ are said to be \textit{comparable}, if $\preceq_1 \leq \preceq_2$ or $\preceq_2 \leq \preceq_1,$ and otherwise \textit{incomparable}.
\end{definition}

\noindent
Let us consider the four possible cases of the previous definition.
\begin{enumerate}
\item If both quasi-orderings are orderings, we obtain that $\preceq_1 \subseteq \preceq_2,$ i.e. that the finer ordering $\preceq_1$ is contained in the coarser ordering $\preceq_2.$ \vspace{1mm}

\item If both quasi-orders are induced by valuations, say $v_1$ and $v_2,$ we obtain that $v_1(y) \leq v_1(x)$ implies $v_2(y) \leq v_2(x).$ This coincides with the usual coarsening relation on ring valuations. \vspace{1mm}

\item The case where $\preceq_1$ is a valuation and $\preceq_2$ is an ordering cannot occur; $0 \preceq_1 0 \preceq_1 -1$ would imply $0 \preceq_2 -1,$ a contradiction. Hence, an ordering cannot be coarser than a valuation. \vspace{1mm}

\item If $\preceq_1$ is an ordering and $\preceq_2$ a valuation, we obtain that they are compatible with each other in the usual sense.
\end{enumerate}

\noindent
In particular, Definition \ref{defcoarse} yields the same cases as Efrat's coarsening relation (cf. \cite[Ch. 7]{Efrat}). We conclude this section by proving that $\leq$ defines a partial ordering on $\mathcal{Q}(R).$

\begin{lemma} \label{signs1} Let $R$ be a ring and $\preceq_1,\preceq_2 \; \in \mathcal{Q}(R).$ Then the following are equivalent:
\begin{enumerate}
\item $\preceq_1 \, \leq \, \preceq_2,$ \vspace{1mm}

\item $\forall x,y \in R\colon 0 \preceq_1 x \preceq_1 y \Rightarrow 0 \preceq_2 x \preceq_2 y.$
\end{enumerate}
\end{lemma}
\begin{proof} Obviously, (2) implies (1). For the converse let $\preceq_1 \, \leq \, \preceq_2$ and $0 \preceq_1 x \preceq_1 y.$ Then $x \preceq_2 y.$ Moreover, $0 \preceq_1 0 \preceq_1 x$ yields $0 \preceq_2 x.$ Thus, $0 \preceq_2 x \preceq_2 y.$
\end{proof}

\begin{lemma} \label{minusey} Let $(R,\preceq)$ be a quasi-ordered ring and $x,y \in R.$ If $x \preceq y \prec 0,$ then $0 \prec -y \preceq -x.$
\end{lemma}
\begin{proof} Clearly $0 \prec -x,-y$ by Lemma \ref{-+}. It remains to show that $-y \preceq -x.$ Assume that $-x \prec -y.$ We know that $y \prec 0 \prec -x,-y,$ and therefore $-x \not\sim y$ and $y \not\sim -y.$ By (QR4) it follows from $x \preceq y$ that $0 \preceq y-x,$ and from $-x \preceq -y$ that $y-x \preceq 0.$ Thus, $y-x \in \mathfrak{q}_{\preceq}.$ This implies $-y \sim -x$ (see Lemma \ref{lem1}), a contradiction. Hence, $-y \preceq -x.$
\end{proof}

\begin{proposition} \label{tree1} Let $R$ be a ring. Then $(\mathcal{Q}(R),\leq)$ is a partially ordered set.
\begin{proof} Obviously, $\leq$ is reflexive. Next, suppose that $\preceq_1 \, \leq \, \preceq_2 \, \leq \, \preceq_3,$ and let $x,y \in R$ such that $0 \preceq_1 x \preceq_1 y.$ From $\preceq_1 \, \leq \, \preceq_2$ and Lemma \ref{signs1} it follows $0 \preceq_2 x \preceq_2 y.$ Thus, $\preceq_2 \, \leq \, \preceq_3$ yields $x \preceq_3 y.$ This shows $\preceq_1 \, \leq \, \preceq_3,$ whence $\leq$ is transitive. \vspace{1mm}

\noindent
Last but not least, we prove that $\leq$ is anti-symmetric. So suppose that $\preceq_1 \, \leq \, \preceq_2$ and $\preceq_2 \, \leq \, \preceq_1.$ Then Lemma \ref{signs1} implies that $0 \preceq_1 x \preceq_1 y \Leftrightarrow 0 \preceq_2 x \preceq_2 y$ for all $x,y \in R,$ in particular $0 \preceq_1 x \Leftrightarrow 0 \preceq_2 x.$ So if $\preceq_1 \, \neq \, \preceq_2,$ there exist some $x,y \in R$ such that (w.l.o.g.) $x \preceq_1 y \prec_1 0,$ but $y \prec_2 x \prec_2 0.$ Lemma \ref{minusey} tells us that $0 \prec_1 -y \preceq_1 -x,$ and $0 \prec_2 -x \preceq_2 -y.$ Now if $-x \sim_2 -y,$ then \cite[Corollary 3.10]{Kuhlmul} implies $x \sim_2 y,$ a contradiction. Hence, $0 \prec_1 -y \preceq_1 -x$ but $-x \prec_2 -y,$ a contradiction to the assumption $\preceq_1 \, \leq \, \preceq_2.$ Therefore, $\preceq_1 \, = \, \preceq_2.$
\end{proof}
\end{proposition}

\section{A Tree Structure Theorem for Quasi-Ordered Rings} \label{sectiontree}

\noindent
In the present section we establish the main theorem of this paper, stating that for any ring $R$ and any two-sided completely prime ideal $\mathfrak{q}$ of $R,$ the set $(\mathcal{Q}_{\mathfrak{q}}(R),\leq)$ of all quasi-orderings on $R$ with support $\mathfrak{q},$ equipped with the coarsening relation $\leq,$ is a rooted tree (Theorem \ref{treering}). From this we deduce a partial ordering $\leq',$ which canonically partitions $\mathcal{Q}(R)$ into the trees $\mathcal{Q}_{\mathfrak{q}}(R)$ (Corollary \ref{forest}). Next, we show that a quasi-ordering $\preceq$ on $R$ is comparable with a quasi-ordering of support $\mathfrak{q}$ if and only if $\mathfrak{q}$ is $\preceq$-convex (Proposition \ref{qcomp}). This gives rise to the result that the set of all valuations with support $\mathfrak{q}$ and all orderings for which $\mathfrak{q}$ is convex, partially ordered by $\leq,$ is a rooted tree (Proposition \ref{convtree}). \vspace{1mm}

\noindent
In this section, let $R$ always be a ring and $\mathfrak{q}$ a two-sided completely prime ideal of $R.$ Moreover, we denote by $\mathcal{Q}_{\mathfrak{q}}(R)$ the set of all quasi-orderings on $R$ with support $\mathfrak{q.}$ Finally, $\leq$ always denotes the coarsening relation introduced in Definition \ref{defcoarse}.

\begin{definition} \label{tree} \hspace{7cm}
\begin{enumerate}
\item A partially ordered set $(S,\leq)$ is called a \textit{tree} if for all $s,s_1,s_2 \in S$ with $s \leq s_1,s_2,$ the elements $s_1$ and $s_2$ are comparable, i.e. $s_1 \leq s_2$ or $s_2 \leq s_1.$ \vspace{1mm}

\item A \textit{rooted tree} is a tree that admits a maximum. \vspace{1mm}

\item A \textit{branch} $S_i$ of a tree $S$ is a maximal chain in $S.$ \vspace{1mm}
 \end{enumerate} 
\end{definition}

\begin{remark} \label{reverse} Different to the usual definition of a (rooted) tree, we reversed the ordering in Definition \ref{tree}. This is because we want coarse valuations to be large elements with respect to our coarsening relation. In particular, trivial valuations shall be maximal elements. \vspace{1mm}
\end{remark}

\noindent
Unfortunately, as we will see below (Lemma \ref{counterexp}), the whole set $(\mathcal{Q}(R),\leq)$ is in general neither a tree, nor admits a maximum. We first show that $(\mathcal{Q}_{\mathfrak{q}}(R),\leq)$ is always a rooted tree (see Theorem \ref{treering}).

\begin{lemma} \label{iff} Let $\preceq_1,\preceq_2 \, \in \mathcal{Q}_{\mathfrak{q}}(R)$ such that $\preceq_1 \, \leq \, \preceq_2$ and $-1 \preceq_2 0.$ Then $0 \preceq_1 x$ if and only if $0 \preceq_2 x.$
\end{lemma}
\begin{proof} The only-if-implication is trivial. For the if-implication suppose that $0 \preceq_2 x,$ and assume for the sake of a contradiction that $x \prec_1 0.$ Then $0 \prec_1 -x,$ so from the assumption $\preceq_1 \, \leq \, \preceq_2$ and the fact that the supports of these quasi-orderings coincide, we obtain $0 \prec_2 -x.$ Applying (QR3) to $0 \preceq_2 x$ with $-x$ yields $0 \preceq_2 -1,$ a contradiction to the assumption that $-1 \preceq_2 0.$ Thus, $0 \preceq_1 x.$
\end{proof}

\begin{lemma} \label{mult} Let $0 \prec x \preceq y$ and $0 \preceq a \prec b.$ Then $ax \prec by$ and $xa \prec yb.$
\end{lemma} 
\begin{proof} Using (QR2) and (QR3), we obtain from $a \prec b$ and $0 \prec x$ that $ax \prec bx.$ Moreover, from $x \preceq y$ and $0 \preceq b,$ it follows via (QR2) that $bx \preceq by.$ Therefore, $ax \prec by.$ The same arguing shows that also $xa \prec yb.$
\end{proof}

\begin{proposition} \label{tree3} Let $\preceq,\preceq_1,\preceq_2 \, \in \mathcal{Q}_{\mathfrak{q}}(R)$ such that $\preceq \, \leq \, \preceq_1,\preceq_2.$ Then $\preceq_1$ and $\preceq_2$ are comparable.
\end{proposition}
\begin{proof} First suppose that $0 \prec_i -1$ for $i= 1,2.$ By Lemma \ref{positive} and Corollary \ref{cor10}(2), we know that $0 \preceq_i x$ and $x \sim_{\preceq_i} -x$ for all $x \in R.$ Assume for the sake of a contradiction that $\preceq_1$ and $\preceq_2$ are incomparable. Then we find $a_i,b_i \in R\backslash \mathfrak{q}$ such that $0 \prec_1 a_1 \preceq_1 a_2$ but $0 \prec_2 a_2 \prec_2 a_1,$ and $0 \prec_2 b_1 \preceq_2 b_2$ but $0 \prec_1 b_2 \prec_1 b_1.$ Since $0 \prec_i -1,$ we may w.l.o.g. assume that $0 \prec a_i, b_i$ (otherwise replace $a_i$ with $-a_i,$ respectively $b_i$ with $-b_i$). Lemma \ref{mult} implies $a_1b_2 \prec_1 a_2b_1$ and $a_2b_1 \prec_2 a_1b_2.$ Since all $a_i,b_i$ are positive w.r.t. $\preceq,$ the contraposition of $\preceq \, \leq \, \preceq_i$ yields $a_1b_2 \prec a_2b_1$ and $a_2b_1 \prec a_1b_2,$ a contradiction. Hence, $\preceq_1$ and $\preceq_2$ are comparable. \vspace{1mm}

\noindent
Next, suppose that $-1 \prec_i 0$ for some $i.$ Then also $-1 \prec 0,$ because $\preceq \, \leq \, \preceq_i.$ Lemma \ref{iff} implies $0 \preceq x \Leftrightarrow 0 \preceq_i x.$ Together with Lemma \ref{-+} (if $y \notin \mathfrak{q}),$ respectively Lemma \ref{lem1} (if $y \in \mathfrak{q}),$ this yields
\[
x \preceq y \Leftrightarrow x-y \preceq 0 \Leftrightarrow x-y \preceq_i 0 \Leftrightarrow x \preceq_i y,
\]
whence $\preceq \, = \, \preceq_i.$ Since $\preceq \, \leq \, \preceq_1, \preceq_2,$ this finishes the proof.
\end{proof}

\begin{definition} A quasi-ordering $\preceq$ on $R$ is called \textit{trivial} if $R/{\sim}_{\preceq}$ has exactly two elements.
\end{definition}

\begin{remark} \label{remtriv} By (QR1), a trivial quasi-orderings admits the equivalence classes $[0] = \mathfrak{q}_{\preceq}$ and $[1] = R\backslash\mathfrak{q}_{\preceq}.$ There is a unique trivial quasi-ordering with support $\mathfrak{q},$ namely the one induced by the trivial valuation with support $\mathfrak{q}.$ We denote it by $\preceq_{\mathfrak{q}}.$
\end{remark}

\begin{proposition} \label{tree2} 
The trivial quasi-ordering $\preceq_{\mathfrak{q}}$ is the maximum of $(\mathcal{Q}_{\mathfrak{q}}(R),\leq).$
\begin{proof} Let $\preceq \: \in \mathcal{Q}_{\mathfrak{q}}(R),$ and let $x,y \in R$ such that $0 \preceq x \preceq y.$ If $x \in \mathfrak{q},$ then $x \preceq_{\mathfrak{q}} y$ by Remark \ref{remtriv}. If $x \notin \mathfrak{q},$ then also $y \notin \mathfrak{q}$ since $0 \preceq x \preceq y.$ Consequently, $x \sim_{\preceq_{\mathfrak{q}}} 1 \sim_{\preceq_{\mathfrak{q}}} y.$ Therefore, $\preceq \, \leq \, \preceq_{\mathfrak{q}},$ so $\preceq_{\mathfrak{q}}$ is the maximum of $\mathcal{Q}_{\mathfrak{q}}(R).$
\end{proof}
\end{proposition}

\begin{theorem} \label{treering} $\mathrm{(the \; tree \; structure \; theorem)}$ \\ Let $R$ be a ring and $\mathfrak{q}$ a two-sided completely prime ideal of $R.$ Then $(\mathcal{Q}_{\mathfrak{q}}(R),\leq)$ is a rooted tree.
\begin{proof} By Proposition \ref{tree1}, $(\mathcal{Q}_{\mathfrak{q}}(R),\leq)$ is a partially ordered set. Proposition \ref{tree3} implies that, given a quasi-ordering $\preceq \, \in \mathcal{Q}_{\mathfrak{q}}(R),$ its coarsenings in $\mathcal{Q}_{\mathfrak{q}}(R)$ are linearly ordered. Moreover, $\mathcal{Q}_{\mathfrak{q}}(R)$ admits a maximum, namely the trivial quasi-ordering with support $\mathfrak{q}$ (Proposition \ref{tree2}). Thus, $(\mathcal{Q}_{\mathfrak{q}}(R),\leq)$ is a rooted tree.
\end{proof}
\end{theorem}

\begin{corollary} \label{forest} Let $R$ be a ring. Then the partial ordering
\[
\preceq_1 \, \leq' \preceq_2 \, :\Leftrightarrow \, \preceq_1 \leq \, \preceq_2 \textrm{and } \mathfrak{q}_{\preceq_2} \subseteq \mathfrak{q}_{\preceq_1} 
\]
yields a partition of $\mathcal{Q}(R)$ into the rooted trees $\mathcal{Q}_{\mathfrak{q}}(R),$ where $\mathfrak{q}$ runs over all two-sided completely prime ideals of $R.$ Moreover, $(\mathcal{Q}(R),\leq')$ is a tree.
\end{corollary}
\begin{proof} Note that $\preceq_1 \, \leq' \, \preceq_2$ if and only if $\preceq_1 \, \leq \, \preceq_2$ and $\mathfrak{q}_1 = \mathfrak{q}_2.$ Therefore, both claims easily follow from Theorem \ref{treering}. 
\end{proof}

\begin{example} \label{treeexp1} We give two examples of the tree structure theorem. 
\begin{enumerate}
\item The integers $\mathbb{Z}$ admit as support $\{0\}$ quasi-orderings the unique ordering $\leq,$ the trivial valuation $v_0,$ and the $p$-adic
valuation $v_p$ for any prime number $p \in \mathbb{N}.$ Therefore, the rooted tree $(\mathcal{Q}_{0}(\mathbb{Z}),\leq)$ is completely described by the following diagram, where a connecting line between two quasi-orderings means that the upper quasi-ordering is coarser than the lower one: 
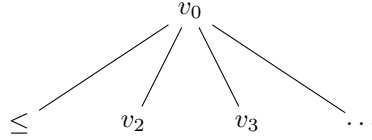
\vspace{1mm}

\begin{center}
\begin{tikzpicture}
\node {$v_0$}
    child{ node {$\leq$}}
		child{ node {$v_2$}}
		child{ node {$v_3$}}
		child{ node {$\ldots$}}         
;
\end{tikzpicture}
\end{center} \vspace{1mm}

\item Consider the following quasi-orderings on $\mathcal{Q}_{0}(\mathbb{Q}[X])\colon$ \vspace{2mm}

\begin{itemize}
\item[(i)] the Archimedean ordering $P_a$ defined by $0 \leq f :\Leftrightarrow 0 \leq f(\pi)$ in $\mathbb{R},$ \vspace{1mm}

\item[(ii)] the non-Archimedean ordering $P_{na}$ defined by $\mathbb{Q} < X,$ \vspace{1mm}

\item[(iii)] the non-Archimedean ordering $P'_{na}$ defined by $X < \mathbb{Q},$ \vspace{1mm}

\item[(iv)] the degree valuation $v_{\textrm{deg}}\colon \mathbb{Q}[X] \to \mathbb{Z} \cup \{\infty\}, \: f \mapsto -\textrm{deg}(f),$ \vspace{1mm}

\item[(v)] the valuation $w\colon \mathbb{Q}[X] \to \mathbb{Z}, \, 0 \neq \sum a_iX^i \mapsto \min\{v_p(a_i)\colon a_i \neq 0\}$ for some $p$-adic valuation $v_p$ on $\mathbb{Q}$ (cf. \cite[Theorem 2.2.1]{Prestel}). 
\end{itemize} \vspace{1mm}

\noindent
Then the respective part of $(\mathcal{Q}_0(\mathbb{Q}[X]),\leq)$ looks as follows: 
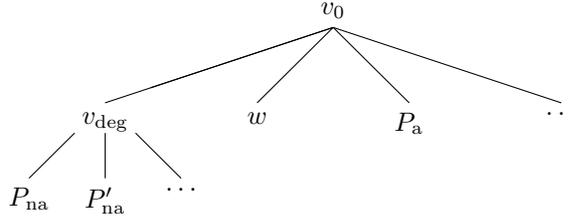
\vspace{1mm}

\begin{center}
\begin{tikzpicture}
\tkzDefPoint(0,0){a}
\tkzDefPoint(-3,-1){b}
\tkzDefPoint(-1,-1){c}
\tkzDefPoint(1,-1){d}
\tkzDefPoint(3,-1){e}
\tkzDefPoint(-4,-2){f}
\tkzDefPoint(-3,-2){g}
\tkzDefPoint(-2,-2){h}
\draw (0,0) -- (-3,-1);
\draw (0,0) -- (-1,-1); 
\draw (0,0) -- (1,-1); 
\draw (0,0) -- (3,-1); 
\draw (0,0) -- (-2.4,-0.8);
\draw (-3.4,-1.4) -- (-4,-2);
\draw (-3,-1.4) -- (-3,-2);
\draw (-2.6,-1.4) -- (-2,-2);
\tkzLabelPoint[above](a){$v_0$}
\tkzLabelPoint[below](b){$v_{\textrm{deg}}$}
\tkzLabelPoint[below](c){$w$}
\tkzLabelPoint[below](d){$P_{\textrm{a}}$}
\tkzLabelPoint[below](e){$\ldots$}
\tkzLabelPoint[below](f){$P_{\textrm{na}}$}
\tkzLabelPoint[below](g){$P'_{\textrm{na}}$}
\tkzLabelPoint[below](h){$\ldots$} 
\end{tikzpicture}
\end{center}
\end{enumerate}
\end{example}
 \vspace{1mm}

\noindent
Next, we discuss the failure of the tree structure theorem for $(\mathcal{Q}(R),\leq).$ \vspace{1mm}

\noindent
The following proposition and Theorem \ref{treering} particularly imply that $(\mathcal{Q}(R),\leq)$ is a rooted tree if and only if $R$ has exactly one two-sided completely prime ideal. \vspace{1mm}

\begin{proposition} \label{minex} Let $R$ be a ring. Then $(\mathcal{Q}(R),\leq)$ admits a maximum if and only if $R$ has exactly one two-sided completely prime ideal.
\begin{proof} The if-implication was already shown in Proposition \ref{tree2}. \vspace{1mm}

\noindent 
Conversely, suppose that $R$ has at least two different such ideals, and assume that $(\mathcal{Q}(R),\leq)$ admits a maximum, say $\preceq.$ Then Proposition \ref{tree2} tells us that $\preceq$ must be trivial, say $\preceq \, = \, \preceq_{\mathfrak{q}}.$ Let $\mathfrak{q} \neq \mathfrak{p}$ be another two-sided completely prime ideal of $R.$ From $\preceq_{\mathfrak{p}} \, \leq \, \preceq_{\mathfrak{q}},$ it follows that $\mathfrak{p} \subsetneq \mathfrak{q}.$ Let $y \in \mathfrak{q}\backslash \mathfrak{p}.$ Then we have $0 \preceq_{\mathfrak{p}} 1 \preceq_{\mathfrak{p}} y,$ but $y \prec_{\mathfrak{q}} 1,$ a contradiction to $\preceq_{\mathfrak{p}} \, \leq \, \preceq_{\mathfrak{q}}.$ Thus, $(\mathcal{Q}(R),\leq)$ admits no maximum.
\end{proof}
\end{proposition}

\begin{remark} A commutative ring $R$ has exactly one prime ideal if and only if every $x \in R$ is either nilpotent or a unit.
\end{remark}

\noindent 
The previous discussion raises the question whether the subset $\{\preceq \, \in \mathcal{Q}(R)\colon {\preceq} \leq \, \preceq_{\mathfrak{q}}\}$ of $\mathcal{Q}(R)$ is a rooted tree. The notion of convexity yields a precise characterisation of the quasi-orderings in $\mathcal{Q}(R)$ that are finer than $\preceq_{\mathfrak{q}}$ (see Proposition \ref{qcomp}). 

\begin{definition} Let $(R,\preceq)$ be a quasi-ordered ring and $M \subseteq R$ an additive subgroup of $R.$ Then $M$ is \textit{convex}, if $0 \preceq x \preceq y \in M$ implies $x \in M$ for all $x,y \in R.$
\end{definition}

\begin{example} \hspace{7cm}
\begin{enumerate}
\item If $(R,\preceq)$ is a quasi-ordered ring, then $\mathfrak{q}_{\preceq}$ is the smallest convex subgroup of $R.$ \vspace{1mm}

\item Let $R = \mathbb{R}[X],$ let $P$ be the unique ordering on $R$ with $0 <_P X <_P \mathbb{R}_{>0},$ and let $M = (X).$ Then $M$ is convex, and $\mathfrak{q}_{P} = \{0\} \subsetneq M.$ \vspace{1mm}

\item Replace $M$ in (2) with $M' = (X+1).$ Then $0 <_P X <_P X+1 \in M',$ but $X \notin M'.$ Thus, $M'$ is not convex.
\end{enumerate}
\end{example}

\begin{proposition} \label{qcomp} Let $R$ be a ring, $\preceq$ a quasi-ordering on $R$ and $\mathfrak{q}$ a two-sided completely prime ideal of $R.$ The following are equivalent:
\begin{enumerate}
\item[$\mathrm{(1)}$] $\mathfrak{q}$ is $\preceq$-convex, \vspace{1mm}

\item[$\mathrm{(2)}$] $\preceq \, \leq \, \preceq_1$ for some quasi-ordering $\preceq_1$ with support $\mathfrak{q},$ \vspace{1mm}

\item[$\mathrm{(3)}$] $\preceq \, \leq \preceq_{\mathfrak{q}}.$ \vspace{1mm}

\end{enumerate}
\begin{proof} We first show that (1) implies (2) and (3). Let $0 \preceq x \preceq y.$ We have to show that $x \preceq_{\mathfrak{q}} y.$ For $x \in \mathfrak{q}$ this is trivial. If $x \notin \mathfrak{q},$ then $y \notin \mathfrak{q}$ by (1), and therefore $x \sim_{\preceq_{\mathfrak{q}}} 1 \sim_{\preceq_{\mathfrak{q}}} y.$
\vspace{1mm}

\noindent
If $\preceq$ is finer than some quasi-ordering $\preceq_1$ with support $\mathfrak{q},$ then $\preceq \, \leq \, \preceq_1 \, \leq \, \preceq_{\mathfrak{q}},$ whence $\preceq \, \leq \, \preceq_{\mathfrak{q}}.$ Thus, (2) implies (3). \vspace{1mm}

\noindent
Finally, we show that (3) implies (1). So suppose that $0 \preceq x \preceq y \in \mathfrak{q}.$ Then also $0 \preceq_{\mathfrak{q}} x \preceq_{\mathfrak{q}} y$ by (3) and Lemma \ref{signs1}. Therefore, $x \in \mathfrak{q}.$
\end{proof}
\end{proposition}

\begin{corollary} \label{convqo} Let $R$ be a ring and $\mathfrak{q}$ a two-sided completely prime ideal of $R.$ Then $$\mathcal{Q}_{\mathfrak{q}}(R)' := \{ \preceq \; \in \mathcal{Q}(R)\colon \mathfrak{q} \textrm{ is } \preceq\textrm{-convex}\},$$ equipped with $\leq,$ is a partially ordered set admitting a maximum, namely $\preceq_{\mathfrak{q}}.$
\end{corollary}
\begin{proof} This follows immediately from Proposition \ref{tree1} and Proposition \ref{qcomp}. 
\end{proof}

\noindent
However, $(\mathcal{Q}_{\mathfrak{q}}(R)',\leq),$ and therefore a fortiori $(\mathcal{Q}(R),\leq),$ is in general not a tree. We may for instance consider the following valuations (cf. \cite[Theorem 2.2.1]{Prestel}): \vspace{1mm}

\begin{itemize}
\item[] $v\colon \mathbb{R}[X,Y] \to \mathbb{Z} \times \mathbb{Z}, \; 0 \neq \sum_{i,j} a_{ij}X^iY^j \mapsto \min\{i(1,0)+j(0,1)\colon a_{ij} \neq 0\},$ \vspace{1mm}
 
\item[] $w\colon \mathbb{R}[X,Y] \to \mathbb{Z}, \; 0 \neq \sum_{i,j} a_{ij}X^iY^j \mapsto \min\{i+j\infty\colon a_{ij} \neq 0\},$ \vspace{1mm}

\item[] $u\colon \mathbb{R}[X,Y] \to \mathbb{Z}, \; 0 \neq \sum_{i,j} a_{ij}X^iY^j \mapsto \min\{j\colon a_{ij} \neq 0 \textrm{ for some } i\},$

\end{itemize} \vspace{1mm}

\noindent
where $\mathbb{Z} \times \mathbb{Z}$ is equipped with the inverse lexicographic ordering.

\begin{lemma} \label{counterexp} Consider the valuations $v,w$ and $u$ on $\mathbb{R}[X,Y].$
\begin{enumerate}
\item $(Y)$ is convex w.r.t. $v,w$ and $u.$ \vspace{1mm}

\item $v \leq w,u,$ but neither $w \leq u,$ nor $u \leq w.$ 
\end{enumerate}
\end{lemma}
\begin{proof} In this proof, when we consider polynomials in $\mathbb{R}[X,Y],$ we may w.l.o.g. assume that they are monomials. This is due to the definition of $v,w$ and $u.$ \vspace{1mm}

\begin{enumerate}
\item We have to show that if $f \in (Y)$ and $v(f) \leq v(g),$ then also $g \in (Y)$ for any $g \in \mathbb{R}[X,Y],$ and analogously for $w$ and $u.$ For either of these valuations, this is obviously true by the way they are defined. \vspace{1mm}

\item We first show $v \leq w,u.$ Suppose that $u(g) < u(f).$ Then the power of $Y$ in $f$ is bigger than the one in $g.$ Thus, also $v(g) < v(f).$ Hence, $v \leq u.$ Likewise, suppose that $w(g) < w(f).$ Then $Y$ does not appear in $g,$ and $Y$ appears in $f$ or the power of $X$ in $f$ is strictly bigger than the power of $X$ in $g.$ Hence, $v(g) < v(f).$ Therefore, also $v \leq w.$ \vspace{1mm}

\noindent
Clearly $w \not\leq u,$ since $w \leq u$ would imply that $(Y) = \mathfrak{q}_w \subseteq \mathfrak{q}_u = (0),$ a contradiction. Now let $f = X^2$ and $g = X.$ Then
$w(g) = 1 < 2 = w(f),$ but $u(g) = u(f) = 0,$ i.e. $u(g) \not < u(f).$ Thus, $u \not\leq w.$
\end{enumerate}
\end{proof}

\noindent
The failure of both properties of a rooted tree over $(\mathbb{R}[X,Y],\leq)$ can be visualised by the following diagram, where $v_0$ and $v_{(Y)}$ denote the trivial valuations with support $\{0\},$ respectively $(Y).$

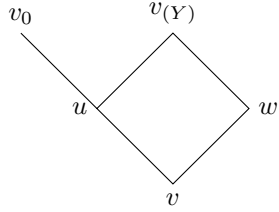
\begin{center}
\begin{tikzpicture}
\tkzDefPoint(1,2){t}
\tkzDefPoint(0,1){u}
\tkzDefPoint(2,1){w}
\tkzDefPoint(1,0){v}
\tkzDefPoint(-1,2){s}
  \draw (-1,2) -- (0,1) -- (1,0) -- (2,1) -- (1,2) -- (0,1);
\tkzLabelPoint[below](v){$v$}
\tkzLabelPoint[left](u){$u$}
\tkzLabelPoint[right](w){$w$}
\tkzLabelPoint[above](s){$v_0$}
\tkzLabelPoint[above](t){$v_{(Y)}$}
\end{tikzpicture}
\end{center}

\noindent
We conclude this section by showing that the statement of Theorem \ref{treering} can nevertheless be improved, however, at the cost of uniformity. More precisely, we will show that the set $\mathcal{Q}_{\mathfrak{q}}(R)''$ consisting of all valuations with support $\mathfrak{q}$ and all $\mathfrak{q}$-convex orderings is a rooted tree.
\vspace{1mm}

\noindent
By Proposition \ref{tree1} and Proposition \ref{qcomp}, it remains to verify that $P \leq \, \preceq_1, \preceq_2$ implies $\preceq_1 \leq \preceq_2$ or $\preceq_2 \leq \preceq_1$ for all $P, \preceq_1,\preceq_2 \, \in \mathcal{Q}(R)_{\mathfrak{q}}''.$ Note that we wrote $P$ for an ordering, because for a valuation $v$ we would have $\mathfrak{q}_v = \mathfrak{q}_{\preceq_1} = \mathfrak{q}_{\preceq_2}$ by definition of $\mathcal{Q}_\mathfrak{q}(R)'',$ whence the claim follows from Theorem \ref{treering}. We distinguish whether $\preceq_1$ and $\preceq_2$ are orderings or valuations. \vspace{1mm}

\noindent
If $\preceq_1$ and $\preceq_2$ are orderings, this is a standard result from real algebra. The result is also known if both, $\preceq_1$ and $\preceq_2,$ are valuations (cf. \cite[Lemma 4.2]{MarshZhang}). Therefore, we are left to deal with the case where $P \leq v,Q$ for some valuation $v$ with support $\mathfrak{q}$ and some orderings $P,Q$ such that $\mathfrak{q}$ is convex with respect to the two of them.

\begin{lemma} Let $R$ be a ring, $v$ a valuation, and $P,Q$ orderings on $R$ such that $P \leq v,Q.$ Then $Q \leq v$ if and only if $\mathfrak{q}_Q \subseteq \mathfrak{q}_v.$
\begin{proof} The only-if-implication is trivial. For the converse we have to show that $0 \leq_Q x \leq_Q y$ implies $v(y) \leq v(x).$ If $x \in \mathfrak{q}_Q \subseteq \mathfrak{q}_v$ this is clear. If $x \notin \mathfrak{q}_Q,$ then also $y \notin \mathfrak{q}_Q.$ Since $P \subseteq Q,$ we have $Q = P \cup \mathfrak{q}_Q,$ whence $x,y \in P.$ Thus, $0 \leq_P x \leq_P y,$ so $P \leq v$ implies $v(y) \leq v(x).$ 
\end{proof}
\end{lemma}

\enlargethispage{2mm}

\noindent
By definition of $\mathcal{Q}_{\mathfrak{q}}(R)'',$ we know that $\mathfrak{q}_Q \subseteq \mathfrak{q}_v$ applies. Thus, we have proven:

\begin{proposition} \label{convtree} Let $R$ be a ring and $\mathfrak{q}$ a two-sided completely prime ideal of $R.$ Then the set $(\mathcal{Q}_\mathfrak{q}(R)'',\leq)$ is a rooted tree.
\end{proposition}

\begin{corollary} \label{xyz} Let $R$ be a ring and $\mathfrak{q}$ a two-sided completely prime ideal of $R.$ Then $(\mathcal{Q}_\mathfrak{q}(R)''',\leq)$ is a tree, where $\mathcal{Q}_\mathfrak{q}(R)'''$ denotes the set of all valuations with support $\mathfrak{q}$ and all orderings with support contained in $\mathfrak{q}.$
\end{corollary}
\begin{proof} This applies, because for the proof of the previous proposition we only used convexity of $\mathfrak{q}$ to ensure that $(\mathcal{Q}_{\mathfrak{q}}(R)'',\leq)$ has a maximum
\end{proof}


\begin{summary} \label{summary} We briefly summarise our results. Let $R$ be a ring, $\mathfrak{q}$ a two-sided completely prime ideal of $R,$ and $\leq$ the coarsening relation from Definition \ref{defcoarse}.
\begin{enumerate}
\item $(\mathcal{Q}(R),\leq)$ is a partially ordered set, where $\mathcal{Q}(R)$ denotes the set of all quasi-orderings on $R$ (Proposition \ref{tree1}). \vspace{1mm}

\item $(\mathcal{Q}_{\mathfrak{q}}(R),\leq)$ is a tree, where $\mathcal{Q}_{\mathfrak{q}}(R)$ denotes the set of all quasi-orderings on $R$ with support $\mathfrak{q}$ (Theorem \ref{treering}). \vspace{1mm}

\item The set $(\mathcal{Q}_{\mathfrak{q}}(R)',\leq)$ of all $\mathfrak{q}$-convex quasi-orderings is a partially ordered set admitting $\preceq_{\mathfrak{q}}$ as a maximum (Corollary \ref{convqo}). \vspace{1mm}

\item The set $(\mathcal{Q}_{\mathfrak{q}}(R)'',\leq)$ of all valuations on $R$ with support $\mathfrak{q}$ and all $\mathfrak{q}$-convex orderings on $R$ is a rooted tree (Proposition \ref{convtree}) \vspace{1mm}

\item The set $(\mathcal{Q}_{\mathfrak{q}}(R)''',\leq)$ of all valuations on $R$ with support $\mathfrak{q}$ and all orderings on $R$ with support contained in $\mathfrak{q}$ is a tree. (Corollary \ref{xyz}).
\end{enumerate} \vspace{1mm}

\noindent
We have
$$\mathcal{Q}(R) \supseteq \mathcal{Q}_{\mathfrak{q}}(R)', \mathcal{Q}_{\mathfrak{q}}(R)''' \supseteq \mathcal{Q}_{\mathfrak{q}}(R)'' \supseteq \mathcal{Q}_{\mathfrak{q}}(R).$$
\end{summary}

\section{On the Tree Structure of Special and Manis Quasi-Orderings} \label{sectiontree2}

\noindent
In Section \ref{sectionspecialmanis} we introduced special quasi-orderings (Definition \ref{specialqo}) and Manis quasi-orderings (Definition \ref{manisqo}). Here, we establish the result that the set $\mathcal{Q}_{\mathfrak{q}}^S(R)$ of all special quasi-orderings with support $\mathfrak{q}$ on a ring $R$ is a rooted subtree of $\mathcal{Q}_{\mathfrak{q}}(R),$ and that the set $\mathcal{Q}^S(R)$ of all special quasi-orderings on $R$ is the ordered disjoint union of the rooted trees $\mathcal{Q}^S_{\mathfrak{q}}(R)$ (Theorem \ref{subtree}). We conclude this section by proving that the same applies to Manis quasi-orderings (Theorem \ref{subtree2}).

\begin{notation} Let $R$ be a ring and $\mathfrak{q}$ a two-sided completely prime ideal of $R.$ We denote by $\mathcal{Q}^S(R)$ the set of all special quasi-orderings on $R,$ and by $\mathcal{Q}_\mathfrak{q}^S(R)$ the set of all special quasi-orderings on $R$ with support $\mathfrak{q}.$ \vspace{1mm}

\noindent
The sets $\mathcal{Q}^M(R)$ and $\mathcal{Q}_{\mathfrak{q}}^M(R)$ are defined analogously for Manis quasi-orderings.
\end{notation}

\begin{definition} A (rooted) \textit{subtree} of a tree $(S,\leq_S)$ is a (rooted) tree $(T,\leq_T)$ such that 
\begin{enumerate}
\item $T \subseteq S$ and $\leq_T \; = \; \leq_S \cap \; T^2,$ \vspace{1mm}

\item $T$ is upward closed under $\leq_S,$ i.e. if $s,t \in S$ and $T \ni s < t,$ then also $t \in T.$
\end{enumerate}
\end{definition}

\noindent
Condition (1) of the previous Definition is obviously satisfied by the subsets $\mathcal{Q}_{\mathfrak{q}}^S(R)$ and $\mathcal{Q}_{\mathfrak{q}}^M(R)$ of $\mathcal{Q}_{\mathfrak{q}}(R)$ The following Lemma shows that (2) also applies.

\begin{lemma} \label{interplay} Let $R$ be a ring, and let $\preceq_1 \, \leq \, \preceq_2$ be quasi-orderings on $R.$ \vspace{1mm}

\begin{itemize}
\item[(1)] If $\preceq_1$ is special, then $\preceq_2$ is special. \vspace{1mm}

\item[(2)] If $\preceq_1$ is Manis, then $\preceq_2$ is Manis. 

\end{itemize}
\end{lemma}

\begin{proof} We show (1), the proof of (2) being analogue. From $\preceq_1 \, \leq \, \preceq_2$ it immediately follows that $\mathfrak{q}_{\preceq_1} \subseteq \mathfrak{q}_{\preceq_2}.$ Now let $x \in R\backslash \mathfrak{q}_2.$ Then also $x \in R \backslash \mathfrak{q}_1.$ So we find some $y \in R$ such that $0 \prec_1 1 \preceq_1 xy.$ Thus, $\preceq_1 \, \leq \, \preceq_2$ implies that $1 \preceq_2 xy,$ i.e. that $\preceq_2$ is special. 
\end{proof}

\begin{definition} If a partially ordered set $(X,\leq_X)$ is the disjoint union of partially ordered sets $(X_i,\leq_{X_i})_{i \in I},$ we say that $X$ is the \textit{ordered disjoint union} of the sets $X_i,$ if for all $x,y \in X\colon$
\[ 
x \leq_X y \Leftrightarrow \exists i \in I\colon x,y \in X_i \textrm{ and } x \leq_{X_i} y.
\]
\end{definition}

\begin{lemma} \label{specsup} Let $R$ be a ring and let $\preceq_1,\preceq_2$ be quasi-orderings on $R$ such that $\preceq_1$ is a special quasi-ordering. If $\preceq_1 \, \leq \, \preceq_2,$ then $\mathfrak{q}_{\preceq_1} = \mathfrak{q}_{\preceq_2}.$
\end{lemma}
\begin{proof} Assume not. Then $\mathfrak{q}_{\preceq_1} \subsetneq \mathfrak{q}_{\preceq_2}.$ So we find some $x \in \mathfrak{q}_{\preceq_2}\backslash\mathfrak{q}_{\preceq_1}.$ Since $\preceq_1$ is special, there is some $y \in R$ such that $0 \prec_1 1 \preceq_1 xy.$ Hence, $\preceq_1 \, \leq \, \preceq_2,$ implies $1 \preceq_2 xy.$ On the other hand, $xy \sim_{\preceq_2} 0,$ since $x \in \mathfrak{q}_{\preceq_2},$ a contradiction. Therefore, $\mathfrak{q}_{\preceq_1} = \mathfrak{q}_{\preceq_2}.$
\end{proof}

\begin{theorem} \label{subtree} Let $R$ be a ring, and let $\mathfrak{q}$ be a two-sided completely prime ideal of $R.$ Then
\begin{enumerate}
\item $(\mathcal{Q}^S_{\mathfrak{q}}(R),\leq)$ is a rooted subtree of $(\mathcal{Q}_{\mathfrak{q}}(R),\leq).$ \vspace{1mm}

\item $(\mathcal{Q}^S(R),\leq)$ is the ordered disjoint union of the rooted trees $(\mathcal{Q}^S_{\mathfrak{q}}(R),\leq).$ 
\end{enumerate}
\end{theorem}
\begin{proof} The set $\mathcal{Q}_{\mathfrak{q}}^S(R) = \mathcal{Q}^S(R) \cap \mathcal{Q}_{\mathfrak{q}}(R)$ is clearly a rooted tree, since $\mathcal{Q}_{\mathfrak{q}}(R)$ is a rooted tree (Theorem \ref{treering}), and since the trivial quasi-ordering $\preceq_{\mathfrak{q}}$ is special. Moreover, by Lemma \ref{interplay}(1), $\mathcal{Q}_{\mathfrak{q}}^S(R)$ is upward closed under $\leq.$ Hence, $(\mathcal{Q}_{\mathfrak{q}}^S(R),\leq)$ is a rooted subtree of $(\mathcal{Q}_{\mathfrak{q}}(R),\leq).$ \vspace{1mm}

\noindent
Obviously, $\mathcal{Q}^S(R)$ is the disjoint union of the trees $\mathcal{Q}_{\mathfrak{q}}^S(R).$ Moreover, if $\preceq_1 \leq \preceq_2$ are in $\mathcal{Q}^S(R),$ then Lemma \ref{specsup} yields $\mathfrak{q}_{\preceq_1} = \mathfrak{q}_{\preceq_1}.$ Hence, $\preceq_1, \preceq_2 \, \in \mathcal{Q}_{\mathfrak{q}}^S(R)$ for some two-sided completely prime ideal $\mathfrak{q}$ of $R.$ Therefore, $\mathcal{Q}^S(R)$ is the ordered disjoint union of the rooted subtrees $\mathcal{Q}^S_{\mathfrak{q}}(R)$ of $\mathcal{Q}_{\mathfrak{q}}(R).$
\end{proof}

\noindent
With the very same methods as in the proof of Theorem \ref{subtree}, we also get:

\begin{theorem} \label{subtree2} Let $R$ be a ring and let $\mathfrak{q}$ be a two-sided completely prime ideal of $R.$ Then
\begin{enumerate}
\item $(\mathcal{Q}^M_{\mathfrak{q}}(R),\leq)$ is a rooted subtree of $(\mathcal{Q}^S_{\mathfrak{q}}(R),\leq),$ respectively $(\mathcal{Q}_{\mathfrak{q}}(R),\leq).$ \vspace{1mm}

\item $(\mathcal{Q}^M(R),\leq)$ is the ordered disjoint union of the rooted trees $(\mathcal{Q}^M_{\mathfrak{q}}(R),\leq).$ 
\end{enumerate}
\end{theorem}

\noindent
The Theorems \ref{subtree} and \ref{subtree2} are an improvement over Corollary \ref{forest} in the sense that for special and Manis quasi-orderings the partition into the rooted trees $\mathcal{Q}_{\mathfrak{q}}^S(R),$ respectively $\mathcal{Q}_{\mathfrak{q}}^M(R),$ comes naturally. Recall that in the said corollary, the partial ordering $\leq'$ artificially enforced the supports to be equal.

\section{Applications of the Tree Structure Theorem} \label{sectionappl}

\noindent
As the main application of the tree structure theorem (Theorem \ref{treering}) we obtain that the trees $(\mathcal{Q}_{\mathfrak{q}}(R),\leq), (\mathcal{Q}(R),\leq'), (\mathcal{Q}^S(R),\leq),$ and $(\mathcal{Q}^M(R),\leq)$ are all spectral sets (Corollary \ref{corspectral}, Corollary \ref{spectral}, Theorem \ref{spectral2}), i.e. isomorphic to the spectrum of a commutative ring, partially ordered by inclusion. Hence, even if we start off with a non-commutative ring $R,$ its quasi-orderings, respectively its special and Manis quasi-orderings, can be realised by the prime ideals of some commutative ring. \vspace{1mm} 

\noindent 
Another, very immediate, application of Theorem \ref{treering} is that dependency is an equivalence relation on each of the sets $\mathcal{Q}_{\mathfrak{q}}(R),$ $\mathcal{Q}^S(R)$ and $\mathcal{Q}^M(R)$ (Proposition \ref{depequiv} and Corollary \ref{depequivspec}), where two quasi-orderings are called dependent if they admit a non-trivial common coarsening.

\begin{definition} \label{spectralpo} A partially ordered set $(X,\leq)$ is called \textit{spectral}, if it is order-isomorphic to $(\textrm{Spec}(S),\subseteq)$ for some commutative ring $S.$
\end{definition}

\noindent
It is a known result that any finite partially ordered set is spectral (cf. \cite{Hochster}, \cite[Theorem 2.10]{Lewis}). In \cite[Corollary 3.3]{LewisOhm}, Lewis and Ohm even showed that a partially ordered set $(X,\leq)$ is spectral, if for any $x \in X$ there are only finitely many $y \in X$ such that $y \leq x,$ respectively $x \leq y.$ This immediately implies:

\begin{proposition} Let $R$ be a ring and $X \subseteq \mathcal{Q}(R)$ such that any quasi-ordering $\preceq \, \in X$ has only finitely many coarsenings and refinements. Then $(X,\leq)$ is spectral.
\end{proposition}

\noindent
In the general case, our work relies on a further result by these authors. They also showed that a partially ordered set $(X,\leq)$ is a tree satisfying Kaplansky's properties (cf. \cite{Kaplansky}, \cite[Ch. 3]{Lewis}, or (K1) and (K2) from Proposition \ref{Kaplansky} below) if and only if $(X,\leq)$ is isomorphic to the spectrum of a B\'{e}zout domain (\cite[Theorem 4.2]{LewisOhm}), i.e. a commutative domain in which the sum of two principal ideals is again a principal ideal.

\begin{proposition} \label{Kaplansky} Let $R$ be a ring and let $\mathfrak{q}$ be a two-sided completely prime ideal of $R.$ The tree $\mathcal{Q}_{\mathfrak{q}}(R)$ has the following properties:
\begin{enumerate}
\item[$\mathrm{(K1)}$] Every chain in $\mathcal{Q}_{\mathfrak{q}}(R)$ has a supremum and an infimum. \vspace{1mm}

\item[$\mathrm{(K2)}$] If $\preceq_1 < \, \preceq_2 \; \in \mathcal{Q}_{\mathfrak{q}}(R),$ then we find elements $\preceq_3,\preceq_4 \; \in \mathcal{Q}_{\mathfrak{q}}(R)$ such that \vspace{1mm}

\begin{itemize}
\item[(i)] $\preceq_1 \, \leq \, \preceq_3 \, < \, \preceq_4 \, \leq \, \preceq_2,$ \vspace{1mm}

\item[(ii)] there is no element in $\mathcal{Q}_{\mathfrak{q}}(R)$ lying strictly in-between $\preceq_3$ and $\preceq_4.$
\end{itemize}
\end{enumerate}
\end{proposition}
\begin{proof} We first prove (K1). Let $(\preceq_i)_{i \in I}$ be a chain in $\mathcal{Q}_{\mathfrak{q}}(R),$ w.l.o.g. $\#I \geq 2.$ It is easy to see that the supremum is given by the union of all quasi-orderings such that $0 \prec_i -1$ (i.e. such that $\preceq_i$ is a valuation), and the infimum either by the only ordering (if $(\preceq_i)_{i \in I}$ contains an ordering) or else by the intersection of all $\preceq_i.$ \vspace{1mm}

\noindent
For (K2) let $\preceq_1 < \, \preceq_2 \, \in \mathcal{Q}_{\mathfrak{q}}(R).$ Define
\[
A := \{\preceq \; \in \mathcal{Q}_{\mathfrak{q}}(R)\colon {\preceq_1} {<} \, {\preceq} \, {\leq} \, {\preceq_2}\}.
\]
Then $(A,\leq)$ is a non-empty partially ordered set. Moreover, for any chain in $A,$ its union is an upper bound in $A.$ Thus, by Zorn's lemma, there is a maximal element $\preceq \; \in A.$ Hence, the choice $\preceq_3 \; = \; \preceq_1$ and $\preceq_4 \; = \; \preceq$ fulfils condition (K2).
\end{proof}

\begin{corollary} \label{corspectral} Let $R$ be a ring and $\mathfrak{q}$ a two-sided completely prime ideal of $R.$ Then $(\mathcal{Q}_{\mathfrak{q}}(R),\leq)$ is a spectral set and anti-isomorphic to the spectrum of a B\'{e}zout domain.
\begin{proof} The second statement follows from Theorem \ref{treering}, Proposition \ref{Kaplansky} and \cite[Theorem 4.2]{LewisOhm}. Recall that in our definition of a tree we reversed the ordering (see Remark \ref{reverse}), which is why we obtain an anti-isomorphism. \vspace{1mm}

\noindent
Now \cite[Proposition 8]{Hochster} states that a partially ordered set $(X,\leq)$ is spectral if and only if $(X,\leq^{0})$ is spectral, where $x \leq^{0} y :\Leftrightarrow y \leq x.$ Therefore, the first claim immediately follows from the second one. 
\end{proof}
\end{corollary} 

\begin{corollary} \label{spectral} Let $R$ be a ring Then $(\mathcal{Q}(R),\leq')$ is a spectral set, where $\leq'$ denotes the partial ordering introduced in Corollary \ref{forest}. Moreover, $(\mathcal{Q}(R),\leq')$ is anti-isomorphic to the spectrum of a B\'{e}zout domain.
\begin{proof} This is proven exactly like Corollary \ref{corspectral}, since Proposition \ref{Kaplansky} obviously also applies to the tree $(\mathcal{Q}(R),\leq').$
\end{proof}
\end{corollary}

\noindent
Proposition \ref{Kaplansky} is easily seen to apply to the trees $(\mathcal{Q}^S(R),\leq)$ and $(\mathcal{Q}^M(R),\leq)$ as well. Therefore, we further obtain:

\begin{theorem} \label{spectral2} Let $R$ be a ring. Then the trees $(\mathcal{Q}^S(R),\leq)$ and $(\mathcal{Q}^M(R),\leq)$ are spectral sets and anti-isomorphic to the spectrum of a B\'{e}zout domain
\begin{proof} This is proven just like Corollary \ref{corspectral}.
\end{proof}
\end{theorem} 

\noindent
The previous theorem also holds for $(\mathcal{Q}_{\mathfrak{q}}^S(R),\leq),$ respectively $(\mathcal{Q}_{\mathfrak{q}}^M(R),\leq).$ More generally, it applies to any tree of quasi-orderings that fulfils the two Kaplansky conditions (K1) and (K2). \vspace{2mm}

\noindent
We continue by introducing a dependency relation on the set $\mathcal{Q}(R).$ As an easy consequence of the tree structure theorem (Theorem \ref{treering}) we obtain that it defines an equivalence relation on $\mathcal{Q}_{\mathfrak{q}}(R)$ for any two-sided completely prime ideal $\mathfrak{q}$ of $R$ (Proposition \ref{depequiv}). The same applies to the sets $\mathcal{Q}^S(R)$ and $\mathcal{Q}^M(R)$ (Corollary \ref{depequivspec}). 

\begin{definition} \label{defdependent} Let $R$ be a ring. Two quasi-orderings $\preceq_1, \preceq_2$ on $R$ are called \textit{dependent}, written $\preceq_1 \, \sim \, \preceq_2,$ if there is a non-trivial quasi-ordering $\preceq$ on $R$ such that $\preceq_1,\preceq_2 \, \leq \, \preceq,$ or if $\preceq_1 \, = \, \preceq_2$ is trivial. Otherwise, we call them \textit{independent}.
\end{definition}

\begin{proposition} \label{depequiv} Let $R$ be a ring and $\mathfrak{q}$ a two-sided completely prime ideal of $R.$ Then dependency is an equivalence relation on $\mathcal{Q}_{\mathfrak{q}}(R).$
\begin{proof} Reflexivity and symmetry are both trivial. For transitivity suppose that $\preceq_1 \, \sim \, \preceq_2$ and $\preceq_2 \, \sim \, \preceq_3,$ w.l.o.g. all non-trivial. Then there are non-trivial quasi-orderings $\preceq, \preceq' \, \in \mathcal{Q}_{\mathfrak{q}}(R)$ such that $\preceq_1,\preceq_2 \, \leq \, \preceq$ and $\preceq_2, \preceq_3 \, \leq \, \preceq'.$ Hence, $\preceq_2 \, \leq \, \preceq, \preceq'.$ Theorem \ref{treering} implies that (w.l.o.g.) $\preceq \, \leq \, \preceq'.$ Thus, $\preceq_1, \preceq_3 \, \leq \, \preceq',$ i.e. $\preceq_1 \, \sim \, \preceq_3.$
\end{proof}
\end{proposition}

\begin{corollary} \label{depequivspec} Let $R$ be a ring. Then dependency is an equivalence relation on $\mathcal{Q}^S(R)$ and $\mathcal{Q}^M(R).$
\end{corollary}
\begin{proof} This follows immediately from Lemma \ref{specsup} and Proposition \ref{depequiv}.
\end{proof}

\section{The Quasi-Real Spectrum of a Ring} \label{sectionspec}

\noindent
Let $R$ be a ring. In this final section of the present paper we first discuss the topology $\tau_T$ on $\mathcal{Q}(R),$ which is induced by the coarsening relation $\leq$ (Definition \ref{Treetopology} and Proposition \ref{inducedpo2}). This topology coincides with canonical topologies on the real spectrum, respectively on the valuation spectrum (Remark \ref{restriction}). We show that $(\mathcal{Q}(R),\tau_T)$ is a $T_0$ space (Corollary \ref{treekol}). 
\vspace{1mm}

\noindent
Afterwards, we naturally extend the Harrison topology $\mathcal{H}$ from the real spectrum to the quasi-real spectrum $\mathcal{Q}(R)$ (Definition \ref{Harrison}). We prove that $(\mathcal{Q}(R),\mathcal{H})$ is a spectral space for any ring $R,$ i.e. homeomorphic to the spectrum of a commutative ring equipped with the Zariski topology (Theorem \ref{spectralspace}). This provides a uniform proof of the fact that the set of all orderings on $R$ and the set of all valuations on $R$ are both spectral spaces with respect to this topology (Corollary \ref{spectraldicho}). Moreover, we deduce that $(\mathcal{Q}(R),\subseteq)$ is a spectral set (Corollary \ref{spectralsub}). 

\begin{notation} For a ring $R$ we denote by
\begin{itemize}
\item[(i)] $\mathcal{O}(R)$ the set of all orderings on $R.$ \vspace{1mm}

\item[(ii)] $\mathcal{V}(R)$ the set of all valuations on $R.$
\end{itemize}
\end{notation}

\noindent
Let $R$ be a commutative ring. The notion of real spectrum is due to Coste and Roy (\cite{Coste}), and refers to the set $\mathcal{O}(R)$ equipped with the so-called \textit{Harrison topology} (or \textit{spectral topology}), i.e. the topology generated by the open subbasis consisting of all the sets 
\[
U(a,b) := \{\leq \, \in \mathcal{O}(R)\colon a < b\} \quad (a,b \in R)
\]
This notion was later generalised to non-commutative rings by Leung, Marshall and Zhang (\cite{Leung}). Likewise, the term valuation spectrum refers to the set $\mathcal{V}(R)$ equipped with certain topologies. Of our interest is the topology $\tau'$ as introduced by Huber and Knebusch in \cite{Huber}, i.e. the one generated by the open subbasis consisting of all the sets
\[
U'(a,b) := \{v \in \mathcal{V}(R)\colon v(a) > v(b)\} = \{v \in \mathcal{V}(R)\colon a \prec_v b\}.
\]

\begin{definition} \label{Treetopology} Let $R$ be a ring. 
\begin{enumerate}
\item We call $\mathcal{Q}(R)$ the \textit{quasi-real spectrum} of $R.$ \vspace{1mm}

\item For $a,b \in R$ we define the set
\[
W(a,b) := \{ \preceq \; \in \mathcal{Q}(R)\colon 0 \preceq a \preceq b\}.
\] 

\item We call the topology $\tau_T$ on $\mathcal{Q}(R)$ generated by the closed subbasis $$\{W(a,b) \subseteq \mathcal{Q}(R)\colon a,b \in R\}$$ the \textit{tree topology}.
\end{enumerate}
\end{definition}

\noindent
The corresponding open subbasis of $\tau_T$ is given by the sets
\[
W(a,b)^C = \{ \preceq \; \in \mathcal{Q}(R)\colon b \prec a \textrm{ or } a \prec 0\}.
\]

\begin{example} Let $R$ be a ring.
\begin{enumerate}
\item The set $\mathcal{O}(R) = W(-1,0)^C$ is open with respect to $\tau_T.$ \vspace{1mm}

\item The set $\mathcal{V}(R) = W(0,-1)$ is closed with respect to $\tau_T.$
\end{enumerate}
\end{example}

\begin{remark} \label{restriction} \hspace{7cm}
\begin{enumerate}
\item If we restrict $\mathcal{Q}(R)$ to valuations we obtain
\[
W(a,b)^C \cap \mathcal{V}(R) = \{v \in \mathcal{V}(R)\colon b \prec_v a\} = U'(b,a). 
\]
Thus, the tree topology and $\tau'$ coincide on $\mathcal{V}(R).$ \vspace{1mm}

\item If we restrict $\mathcal{Q}(R)$ to orderings we obtain 
\[
W(a,b)^C \cap \, \mathcal{O}(R) = U(b,a) \cup U(a,0), 
\]
and conversely
\begin{align*}
U(a,b) &= \{ \leq \; \in \mathcal{O}(R)\colon a-b < 0 \textrm{ or } 0 < 0\} \\ &= W(0,a-b)^C \cap \, \mathcal{O}(R).
\end{align*}

\noindent
Thus, the tree topology and the Harrison topology coincide on $\mathcal{O}(R).$
\end{enumerate}
\end{remark}

\noindent
The following result describes the relationship between $\leq$ and the tree topology $\tau_T.$

\begin{proposition} \label{inducedpo2} Let $R$ be a ring, let $\preceq_1, \preceq_2 \; \in \mathcal{Q}(R),$ and let $\leq$ denote the coarsening relation on $\mathcal{Q}(R).$ Then
\[
\preceq_2 \, \in \overline{\{\preceq_1\}} \; \Leftrightarrow \; \, \preceq_1 \, \leq \, \preceq_2,
\]
w.r.t. $\tau_T.$
\end{proposition}
\begin{proof} Let $\mathcal{N}_{\preceq_2}$ denote the set of all open neighbourhoods of $\preceq_2$ w.r.t. $\tau_T.$ Then
\begin{align*} \preceq_1 \, \leq \, \preceq_2 &\Leftrightarrow \forall x,y \in R\colon (0 \preceq_1 x \preceq_1 y \Rightarrow 0 \preceq_2 x \preceq_2 y) \\
&\Leftrightarrow \forall x,y \in R\colon (y \prec_2 x \textrm{ or } x \prec_2 0 \Rightarrow y \prec_1 x \textrm{ or } x \prec_1 0) \\
&\Leftrightarrow \forall N \in \mathcal{N}_{\preceq_2}\colon N \cap \{\preceq_1\} \neq \emptyset \\
&\Leftrightarrow \; \preceq_2 \; \in \overline{\{\preceq_1\}}. \end{align*}
\end{proof}

\begin{corollary} \label{treekol} Let $R$ be a ring. Then $(\mathcal{Q}(R),\tau_T)$ is a $T_0$ space. \vspace{1mm}
\end{corollary}
\begin{proof}
Two quasi-orderings $\preceq_1,\preceq_2 \; \in \mathcal{Q}(R)$ are topologically indistinguishable if and only if $\preceq_1 \; \in \overline{\{\preceq_2\}}$ and $\preceq_2 \; \in \overline{\{\preceq_1\}},$ i.e. if and only if $\preceq_1 \, \leq \, \preceq_2 \, \leq \, \preceq_1$ with respect to the coarsening relation $\leq.$ Anti-symmetry of $\leq$ yields $\preceq_1 \, = \, \preceq_2.$ Thus, $(\mathcal{Q}(R),\tau_T)$ is $T_0.$
\end{proof}

\begin{remark} A singleton $\{\preceq\}$ is $\tau_T$-closed if and only if $\preceq$ is a trivial quasi-ordering. This is an immediate consequence of Proposition \ref{inducedpo2} and the fact that the $\leq$-maximal elements are precisely the trivial quasi-orderings.
\end{remark}

\noindent
We conclude this paper by unifying the Harrison topology and the topology $\tau'$ (see above) in a more canonical way.

\begin{definition} \label{Harrison} Let $R$ be a ring. 
\begin{enumerate}
\item For $a,b \in R$ we define the set
\[
U(a,b) := \{ \preceq \; \in \mathcal{Q}(R)\colon a \prec b\}.
\] 
\item The topology $\mathcal{H}$ on $\mathcal{Q}(R)$ generated by the open subbasis $$\{U(a,b) \subseteq \mathcal{Q}(R)\colon a,b \in R\}$$ is called the \textit{Harrison} \textit{topology} (or \textit{spectral topology}). \vspace{1mm}
\end{enumerate}
\end{definition}

\begin{example} \label{expclopen} \hspace{7cm}
\begin{enumerate} 
\item The set $\mathcal{O}(R)$ is open and closed with respect to $\mathcal{H}.$ Using Theorem \ref{dicho} we obtain $\mathcal{O}(R) = U(-1,0)$ and $\left(\mathcal{O}(R)\right)^C = \mathcal{V}(R) = U(0,-1).$ \vspace{1mm}

\item The set $\mathcal{V}(R)$ is also open and closed. This follows immediately from (1).
\end{enumerate}
\end{example}

\noindent
Obviously, by definition, $\mathcal{H}$ coincides with the usual Harrison topology on the real spectrum, and the restriction of $\mathcal{H}$ to $\mathcal{V}(R)$ with the topology $\tau'$ of Huber and Knebusch. In fact, the Harrison topology on $\mathcal{Q}(R)$ shares many properties with the Harrison topology on the real spectrum of a ring. In the following we show that $(\mathcal{Q}(R),\mathcal{H})$ is a spectral space (Theorem \ref{spectralspace}). Our proof of this theorem is (essentially) identical with the one in the real case. 

\begin{lemma} \label{inducedpo} Let $R$ be a ring and $\preceq_1, \preceq_2 \, \in \mathcal{Q}(R).$ Then
\[
\preceq_2 \; \in \overline{\{\preceq_1\}} \; \Leftrightarrow \; \, \preceq_1 \, \subseteq \, \preceq_2, 
\]
w.r.t. $\mathcal{H}.$
\begin{proof} Completely analogue to the proof of Proposition \ref{inducedpo2}.
\end{proof}
\end{lemma}

\begin{corollary} \label{Kolmogorovspace} Let $R$ be a ring. Then $(\mathcal{Q}(R),\mathcal{H})$ is a $T_0$ space.
\begin{proof} Completely analogue to the proof of Corollary \ref{treekol}. 
\end{proof}
\end{corollary}

\noindent
For a further study of the Harrison topology we also generalise the Tychonoff topology (or constructible topology) from the real spectrum to $\mathcal{Q}(R).$

\begin{definition} \label{Tychonoff} Let $R$ be a ring. 
\begin{enumerate}
\item For $a,b \in R$ we define the set
\[
V(a,b) := \{ \preceq \; \in \mathcal{Q}(R)\colon a \preceq b\}.
\] 
\item The topology $\mathcal{T}$ on $\mathcal{Q}(R)$ generated by the open subbasis $$\{U(a,b) \subseteq \mathcal{Q}(R)\colon a,b \in R\} \cup \{V(a,b) \subseteq \mathcal{Q}(R)\colon a,b \in R\} $$ is called the \textit{Tychonoff topology} (or \textit{constructible topology}, respectively \textit{patch topology}). \vspace{1mm}
\end{enumerate}
\end{definition}

\begin{proposition} \label{Boolean} Let $R$ be a ring. Then $(\mathcal{Q}(R),\mathcal{T})$ is a compact space.
\end{proposition}
\begin{proof} We first show that this space is quasi-compact. Any quasi-ordering $\preceq$ on $R$ is uniquely determined by the map 
  \[
     f_{\preceq}\colon R^2 \to \{0,1\}, \: (a,b) \mapsto \left\{\begin{array}{cc} 1, & a \preceq b 
          \vspace{1mm} \\ -1, & b \prec a \end{array}\right.
  \]

\noindent
Hence, we may consider $\mathcal{Q}(R)$ as a subset of $\mathcal{F} := \{0,1\}^{R^2}.$ We equip $\{0,1\}$ with the discrete topology, which makes it a quasi-compact space. So by the Tychonoff theorem, $\mathcal{F}$ is also quasi-compact with respect to the product topology. An open subbasis is given by the sets 
\[
H_{\epsilon}(a,b) := \{f\colon R^2 \to \{0,1\}, \: (a,b) \mapsto \epsilon\} \quad (a,b \in R, \: \epsilon \in \{0,1\}).
\]

\noindent
Therefore, we may identify $\mathcal{Q}(R)$ as a subspace of $\mathcal{F}.$ We show that it is a closed subspace, which implies that $(\mathcal{Q}(R),\mathcal{T})$ is quasi-compact. So let $f \in \mathcal{F}$ be not induced by some quasi-ordering on $R,$ i.e. the induced binary relation on $R$ does not satisfy one of the axioms of a quasi-ordering. If for instance axiom (QR4) fails, then we find some $x,y,z \in R$ such that $f(x,y) = 1$ and $f(y,z) = 0 = f(y+z,x+z),$ respectively $f(z,y) = 0 = f(y+z,x+z).$ Thus, $f$ is separated from $\mathcal{Q}(R)$ by the open set
\[
H_1(x,y) \cap (H_0(y,z) \cup H_0(y,z)) \cap H_0(y+z,x+z),
\]

\noindent
whence $s \notin \overline{\mathcal{Q}(R)}.$ The same arguing applies to all the other axioms. Therefore, $\mathcal{Q}(R)$ is a closed subspace of $\mathcal{F},$ implying that $(\mathcal{Q}(R),\mathcal{T})$ is quasi-compact.

\vspace{1mm}

\noindent
It remains to show that $(\mathcal{Q}(R),\mathcal{T})$ is Hausdorff. So let $\preceq_1, \preceq_2 \; \in \mathcal{Q}(R)$ with $\preceq_1 \; \neq \; \preceq_2.$ Then there exist (w.l.o.g.) some $a,b \in R$ such that $a \preceq_1 b$ and $b \prec_2 a.$ Consequently, $\preceq_1 \, \in V(a,b)\backslash U(b,a)$ and $\preceq_2 \, \in U(b,a)\backslash V(a,b).$ \vspace{1mm}
\end{proof}

\begin{theorem} \label{spectralspace} Let $R$ be a ring. Then $(\mathcal{Q}(R),\mathcal{H})$ is a spectral space, i.e. homeomorphic to the Zariski spectrum of a commutative ring.
\begin{proof} We have already proven that $(\mathcal{Q}(R),\mathcal{H})$ is a $T_0$ space (Corollary \ref{Kolmogorovspace}), and that $(\mathcal{Q}(R),\mathcal{T})$ is a compact space (Proposition \ref{Boolean}). Moreover, the sets $U(a,b)$ are both open and closed with respect to $\mathcal{T},$ and an open subbasis of $\mathcal{H}.$ Therefore, \cite[Proposition 7]{Hochster} tells us that $(\mathcal{Q}(R),\mathcal{H})$ is a spectral space.
\end{proof}
\end{theorem}

\begin{corollary} \label{spectraldicho} Let $R$ be a ring. Then $(\mathcal{O}(R),\mathcal{H})$ and $(\mathcal{V}(R),\mathcal{H})$ are both spectral spaces.
\begin{proof} This is an easy consequence of Theorem \ref{spectralspace} and Example \ref{expclopen}, since closed subsets of spectral spaces are again spectral spaces (cf. \cite[Theorem 2.1.3]{Dick}). 
\end{proof}
\end{corollary}

\begin{corollary} \label{spectralsub} Let $R$ be a ring. Then $(\mathcal{Q}(R), \subseteq)$ is a spectral set.
\begin{proof} By Theorem \ref{spectralspace} there is a commutative ring $S$ such that $(\mathcal{Q}(R), \mathcal{H})$ is homeomorphic to $(\textrm{Spec}(S),\mathcal{Z}),$ where $Z$ denotes the Zariski topology. According to \cite{LewisOhm}, this implies that $\mathcal{Q}(R)$ and $\textrm{Spec}(S)$ are order-isomorphic with respect to the induced partial orderings (see Lemma \ref{inducedpo}). Hence, $(\mathcal{Q}(R), \subseteq) \cong (\textrm{Spec}(S), \subseteq),$ i.e. $(\mathcal{Q}(R), \subseteq)$ is spectral.
\end{proof}
\end{corollary}

\section*{Acknowledgement}
\noindent
The results presented here are part of my doctoral thesis. I am grateful to my supervisor Salma Kuhlmann for her dedicated and continuous support.
\vspace{1mm}

\onehalfspace

\vspace{6mm}

\textsc{Fachbereich Mathematik und Statistik, Universit\"at Konstanz},

78457 \textsc{Konstanz, Germany}, 

E-mail address: simon.2.mueller@uni-konstanz.de

\end{document}